\documentclass[a4paper,11pt]{scrartcl} 

\usepackage[T1]{fontenc}
\usepackage{mathrsfs}
\usepackage{amssymb} 
\usepackage{nicefrac}
\usepackage{mathtools} 
\usepackage{enumitem}
\usepackage{amsthm}
\usepackage{nameref}
\usepackage{comment}
\usepackage{amsmath}
\usepackage{amsfonts}
\usepackage{csquotes}
\usepackage{hyperref,cleveref}
\usepackage{cite}

\hypersetup{
pdfauthor={},
pdftitle={},
breaklinks=true,
colorlinks=true,
linkcolor=blue,
citecolor=blue,
urlcolor=blue,
filecolor=blue,
}

\def\Xint#1{\mathchoice
{\XXint\displaystyle\textstyle{#1}}%
{\XXint\textstyle\scriptstyle{#1}}%
{\XXint\scriptstyle\scriptscriptstyle{#1}}%
{\XXint\scriptscriptstyle\scriptscriptstyle{#1}}%
\!\int}
\def\XXint#1#2#3{{\setbox0=\hbox{$#1{#2#3}{\int}$ }
\vcenter{\hbox{$#2#3$ }}\kern-.6\wd0}}

\def\dashint{\Xint-}

\theoremstyle{plain}
\newtheorem{theorem}{Theorem}[section]
\newtheorem{lemma}[theorem]{Lemma}
\theoremstyle{definition}
\newtheorem{remark}[theorem]{Remark}

\numberwithin{equation}{section}

\newcommand{\bfU}{\mathcal{U}}
\newcommand{\bfZ}{\mathcal{Z}}
\newcommand{\bfY}{\mathcal{Y}}

\newcommand{\bvg}{\bv}
\newcommand{\psig}{\psi}
\newcommand{\mug}{\mu}
\newcommand{\bFg}{\bs{F}}

\newcommand{\dx}{{\,{\mathrm d}x}}
\newcommand{\ds}{{\,{\mathrm d}s}}

\newcommand{\dtheta}{ {\,{\mathrm d}\theta} }

\newcommand{\Wel}{W_{\mathrm{el}}}
\newcommand{\Wpf}{W_{\mathrm{pf}}}

\newcommand{\dxd}{d\times d}

\newcommand{\intO}{\int_{\Omega}}

\newcommand{\TOmega}{{\Omega}_T}

\newcommand{\R}{\mathbb{R}}
\newcommand{\N}{\mathbb{N}}
 
\newcommand{\calF}{\mathscr{F}}

\newcommand{\bbF}{\mathbb{F}}

\newcommand{\D}{{\mathrm D}}

\newcommand{\calV}{\mathcal{V}}

\newcommand{\tdots}{\mathrel{\text{\multiput(0,-2)(0,2){3}{$\cdot$}\,}}}

\newcommand{\dxds}{\,{\mathrm d}(s,x)} 
\newcommand{\dxdt}{\,{\mathrm d}(t,x) }

\newcommand{\bC}{\boldsymbol{C}}

\newcommand{\calE}{\mathscr{E}}

\newcommand{\bv}{\bs{v}}       
\newcommand{\bvD}{\bs{v}_{\mathrm D}}   

\newcommand{\bL}{\bs{L}}  
\newcommand{\bK}{\bs{K}}

\newcommand{\bH}{\bs{H}} 
\newcommand{\bF}{\bs{F}}   
\newcommand{\bA}{\bs{A}} 
\newcommand{\bB}{\bs{B}} 
\newcommand{\bG}{\bs{G}} 
\newcommand{\by}{\bs{y}}  
\newcommand{\bw}{\bs{w}}  

\newcommand{\byg}{\bs{y}}

\newcommand{\DBC}{{\Gamma_{\mathrm D}}}

\newcommand{\tO}{{0}} 

\newcommand{\bs}[1]{\boldsymbol{#1}}   

\newcommand{\I}{\bs{I}_d}    
\newcommand{\id}{\bs{id}}  
 
\newcommand{\Fp}{ \bF_{\mathrm p} } 
 
\newcommand{\Fe}{ \bF_{\mathrm e} }

\DeclareMathOperator{\dv}{div}
\newcommand{\norm}[1]{\lVert #1 \rVert}

\newcommand{\bD}{\boldsymbol{D}}
\DeclareMathOperator{\Div}{div}

\newcommand{\fe}{\calF}
\newcommand{\feel}{\calF_{\mathrm{el}}}
\newcommand{\fepf}{\calF_{\mathrm{pf}}}
\newcommand{\fehy}{\calF_{\mathrm{hy}}}
\newcommand{\Wfe}{W}
\newcommand{\Why}{W_{\mathrm{hy}}}

\begin{document}
\author{
Thomas Eiter\thanks{University of Kassel, Institute of Mathematics, Heinrich-Plett-Str.~40, 34132 Kassel, Germany
} \textsuperscript{,}%
\thanks{Weierstrass Institute for Applied Analysis and Stochastics, Mohrenstr.~39, 10117 Berlin, Germany \\
Email: thomas.eiter@wias-berlin.de\\
\phantom{Email:} leonie.schmeller@wias-berlin.de	% Please use firstname.lastname@wias-berlin.de
}
\and
Leonie Schmeller\footnotemark[2]
}
\title{
Weak solutions to a model for phase separation coupled with finite-strain viscoelasticity subject to external distortion}

\maketitle

\begin{abstract}
We study 
the coupling of
a viscoelastic deformation governed by
a Kelvin--Voigt model at equilibrium,
based on the concept of second-grade nonsimple materials,
with a plastic deformation due to volumetric swelling, 
described via a phase-field variable 
subject to a Cahn--Hilliard model
expressed in a Lagrangian frame.
Such models can be used to describe the time evolution of hydrogels 
in terms of phase separation within a deformable substrate.
The equations are mainly coupled
via a multiplicative decomposition of the deformation gradient
into both contributions
and via a Korteweg term in the Eulerian frame.
To treat time-dependent Dirichlet conditions for the deformation, 
an auxiliary variable with fixed boundary values is introduced,
which results in another multiplicative structure.
Imposing suitable growth conditions on the elastic and viscous potentials,
we construct weak solutions to this quasistatic model as the limit 
of time-discrete solutions to 
incremental minimization problems.
The limit passage is possible due to additional regularity 
induced by the hyperelastic and viscous stresses.
\end{abstract}

\noindent
\textbf{MSC2020:} 
35A01, % Existence problems for PDEs: global existence, local existence, non-existence
35A15, % Variational methods applied to PDEs
35K55, % Nonlinear parabolic equations
35Q74, % PDEs in connection with mechanics of deformable solids
74A30, % Nonsimple materials
74B20% Nonlinear elasticity
\\
\noindent
\textbf{Keywords:} 
Finite-strain elasticity, nonsimple material, hyperelastic stress, Kelvin--Voigt rheology, viscoelasticity, phase-field model, Cahn--Hilliard equation, multiplicative coupling,
incremental minimization
%%%%%%%%%%%%%%%%%%%%%%%%%%%%%%%%%%%%%%%%%%%%%%%%%%%%%%%%%%%%%%%%%%%%%

\section{Introduction}

Mathematical models for coupling multi-phase systems with nonlinear elastic deformation
are fundamental for describing many phenomena in soft matter physics and biology, such as the wetting of soft substrates or the formation of patterns during swelling or deswelling
of hydrogels.
Hydrogels are cross-linked networks of (hydrophilic) polymers dissolved in a liquid,
and they are ubiquitous in nature and have particular applications in medical technology,
for example as scaffolds for cell growth \cite{jockenhoevel2011cardiovascular} and mini-cell encapsulations \cite{fattahi2021core}.
Due to the formation of coexisting regions with different magnitudes of swelling,
mathematical models for hydrogels 
have to
take into account phase-separation processes and nonlinear elastic behavior~\cite{BAM2020,bertrand2016dynamics}. 

In this paper we show existence of weak solutions to 
such a model.
The system is formulated in a Lagrangian configuration,
and the local deformation $\bv$
arises from the combination of plastic deformation due to liquid absorption
with a viscoelastic deformation 
subject to a finite-strain model that allows the consideration of large deformations. 
We consider a two-phase diffuse-interface model, 
where the phase-field variable $\psi$ represents regions with different magnitude of swelling,
are energetically preferred.
The phase-field evolution is subject to a Cahn--Hillard equation,
and the mechanical deformation is described by
a Kelvin--Voigt material model,
thus containing elastic and viscous stresses.
Since the mechanical evolution towards
equilibrium usually happens
at a much faster time-scale than the phase-field evolution,
we consider a quasistatic approximation 
and neglect inertial effects in the mechanical equation.
The elastic potential will be chosen in a way that ensures the absence of
local self-penetration,
and we consider a second-grade nonsimple material,
that is, the stored elastic energy depends on the elastic strain and its gradient.
Such elasticity models were first introduced by Toupin~\cite{toupin1962couplestresses},
and have now become widely accepted in the community~\cite{friedgurtin2006nonsimple,mindlineshel1968firststraingradient,podioguidugli2002nonsimple,triantafyllidisAifantis1986hyperelastic,batra1976nonsimple,mielke2020thermoviscoelasticity,badal2023thermovisc,kruvzik2019mathematical}.
The viscous stress potential will be formulated as a function 
in terms of the right Cauchy--Green tensor in order to ensure frame-indifference,
see also Remark~\ref{rem:frame_indiff} below.

The phase field and the deformation are coupled via a decomposition of the deformation gradient 
into elastic and plastic parts.
Instead of an additive decomposition,
which is suitable for the framework of small strains,
we use a multiplicative decomposition of the deformation gradient.
This concept was first introduced by Kroner~\cite{kroner1959allgemeine} and Lee and Liu~\cite{lee1967finite} in the finite-strain setting
and has become a common approach employed in diverse applications,
also suitable for modeling (isotropic) swelling \cite{abels2022fluid,flory1950statistical,kroner1959allgemeine,lubarda1981correct,lucantonio2013transient,yavari2013nonlinear,mielke2018global,roubicekstefanelli2023swelling,kruvzik2019mathematical,goriely2017biologicalgrowth}.
The plastic deformation depends on the phase field
in such a way that the effect of swelling is purely volumetric and isotropic~\cite[Subsect.~14.2]{goriely2017biologicalgrowth}.

Further, 
we impose time-dependent Dirichlet boundary conditions for the deformation field
on some part of the boundary. 
These boundary conditions can be used to investigate the influence of a variable distortion on pattern formation
during wetting and dewetting of the substrate.
To treat this time dependence,
we follow~\cite{thomas2020analysis,francfort2006existence}
and express the deformation $\bv$ in
terms of a composition with an 
auxiliary variable $\by$ with time-independent boundary values.
As the mechanical equation is expressed in terms of the deformation gradient $\bF=\nabla \bv$, 
this leads to another multiplicative structure in our system. 

The analytical investigation of models
coupling the Cahn--Hilliard equation with elasticity
started few decades ago with the works by
Carrive, Miranville, and Pi\'{e}trus
\cite{carrive2000cahnhilliardelastic},
Garcke~\cite{garcke2000mathematical},
and Miraville~\cite{miranville2000cahnhilliardnonisotropic}.
Since then, the properties of such models were studied 
in various configurations, see~\cite{garcke2005cahnhilliardelasticmisfit,alber2007phaseboundaries,pawlow2009cahnhilliardvisco,heinemann2013existence,heinemann2015existence,garcke2021phasefieldtumour,pawlow2008cahnhilliard}
for example.
While all these articles deal with the small-strain setting
and thus with linear elasticity,
the case of large-strain elasticity, which we are interested in, 
seems to have been treated in only a few articles.
In~\cite{roubicek2021cahn},
related static as well as dynamic models are studied, 
where the latter case 
requires higher-order regularization terms 
and no Dirichlet boundary conditions are considered.
More involved models
considering additional thermodynamic effects were studied in~\cite{roubicek2018magnetoelastic,roubicek2018porousrocks}.
While these works investigate the systems in a Lagrangian framework,
there are also articles using an Eulerian approach, see the recent articles~\cite{agosti2023CHallencahn,agosti2023CHviscoelasticity,roubicekstefanelli2023swelling} and references therein.
For a numerical analysis of large class of biological models
related to the one we study here,
we refer to~\cite{erhardt2023cellinhydrogel,schmellerpeschka2023}.

To show existence of weak solutions to the model considered here,
we first construct approximate solutions
by a monolithic time-incremental scheme. 
We obtain solutions 
to the time-discretized problem 
as solutions to minimization problems
using the direct method of calculus of variations.
Since we consider a hyperelastic stress, which serves as a second-order regularization,
polyconvexity assumptions for the elastic potential are not required.
Moreover, the consideration of a second-grade nonsimple material 
ensures sufficient spatial regularity of the mechanical variable
as well as a uniform lower bound for the determinant of the deformation gradient
using the theory by Healey and Krömer~\cite{healey2009injective} 
and by Mielke and Roub{\'i}{\v{c}}ek \cite{mielke2020thermoviscoelasticity}. 
Since the time-dependent boundary conditions and the coupling to the phase-field lead to 
a non-stationary problem, the mathematical analysis also requires increased time-regularity of the mechanical variable,
which is derived from the viscous stress terms.
This framework allows us to apply a generalized Korn-type inequality,
originally due to Neff~\cite{neff2002korn} and Pompe~\cite{pompe2003korn}, 
which implies the temporal compactness of the mechanical variables, see also \cite{van2023finite,mielke2020thermoviscoelasticity}.
In summary, this means that the hyperelastic and the viscous stress
serve as a regularization, 
which is crucial in order to obtain the weak formulation as the limit
of Euler--Lagrange equations for the discrete minimization problems.
Although these discretized problems will be solved in terms of the 
variables $\by$ and $\psi$, 
the formulation of the (time-discrete) Euler--Lagrange equations
and the weak formulation of the problem will be given in terms of the physical variables $\bv$ and $\psi$, or rather the deformation gradient $\bF=\nabla \bv$ and $\psi$
and their derivatives,
see \eqref{equ:time}, \eqref{equ:chem}, \eqref{equ:elast} below.

\paragraph{Outline}
In Section~\ref{sec:mainresults} we present the mathematical model
studied in this paper
and we state the main result on existence of weak solutions to this system.
In Section~\ref{sec:preliminaries}
we examine
how convergence in the auxiliary variable $\by$ with constant boundary values transfers 
to the actual deformation $\bv$,
and we study regularity properties of the energy functional.
The time-discrete minimization problem 
is introduced in Section~\ref{sec:timediscrete},
and we show existence of minimizers and derive associated a priori bounds
as well as Euler--Lagrange equations.
By passing to the limit in these equations, 
we conclude the proof of existence of weak solutions
in Section~\ref{sec:finalproof}.

\paragraph{Notation}

The symbol $\R_{\geq 0}$ denotes the set of nonnegative real numbers.
We write $|\cdot|$ for the Euclidean norm of a vector in $\R^d$,
a matrix in $\R^{d\times d}$ or a third-order tensor in $\R^{d\times d\times d}$,
where we always consider $d\in\{2,3\}$.
Moreover, ${\mathrm {GL}}_+(d)$ is the subclass of $\R^{d\times d}$
with positive determinant,
and $\mathrm{SO}(d)$ consists of all symmetric matrices with determinant equal to $1$.

Let $X$ be a Banach space with dual space $X^\ast$.
We let $\langle\cdot,\cdot\rangle_{X^\ast\times X}$ be the corresponding dual pairing.
If $X$ is a Hilbert space, we further write $(\cdot,\cdot)_X$
for the associated scalar product.
For $M\subset \R^d$, $d\in\N$ 
an open or closed set,
the space ${\mathrm C}^k(M;X)$, $k\in\N_0$ consists of all $k$-times continuously differentiable maps,
and ${\mathrm C}^{k,\lambda}(M;X)$ denote corresponding H\"older spaces 
with exponent $\lambda\in(0,1)$.
When $X=\R$, we simply write ${\mathrm C}^k(M)$, etc.

By $\Omega$ we always denote a bounded domain in $\R^d$
with Lipschitz boundary $\partial\Omega$
and unit outer normal vector $\nu$.
We write $\partial_t$ and $\partial_j=\partial_{x_j}$, $j=1,\ldots,d$ for
partial derivatives in time and space, respectively.
Moreover, the divergence operator $\dv$ acts on tensors of second or third order
with respect to their last index,
that is,
for $\mathbf A(x)=(A_{ij}(x))\in\R^{d\times d}$ and $\mathbf G(x)=(G_{ijk})\in\R^{d\times d\times d}$,
we have
$(\dv\mathbf A)_i = \partial_j A_{ij}$
and 
$(\dv\mathbf G)_{ij} = \partial_k G_{ijk}$,
where we use Einstein summation convention.

For $p\in[1,\infty]$ and $k\in\N$,
we use the notation $L^p(\Omega)$, $W^{k,p}(\Omega)$, 
and $H^k(\Omega):=W^{k,2}(\Omega)$
for Lebesgue spaces and Sobolev spaces of scalar functions,
and we write $L^p(\Omega;\R^d)$ etc.~for their vector-valued analogs.
Similarly, for a Banach space $X$ and $T>0$, we write 
$L^p(0,T;X)$ and $H^1(0,T;X)$ for Bochner--Lebesgue spaces and Bochner--Sobolev spaces.
Moreover, if $\Gamma_D\subset \partial\Omega$ has positive surface measure $\mathcal{H}^{d-1}(\DBC)>0$,
we set
\[
W^{k,p}_{\DBC}(\Omega)=
\{ u \in W^{k,p}(\Omega) \mid u=0 \text{ on }\DBC\}.
\]

\section{Main results}
\label{sec:mainresults}

We first introduce the specific model we investigate 
in this article,
and we collect the general assumptions on the energy and dissipation potentials, the measure of swelling and the time-dependent Dirichlet data.
After introducing the functional framework,
we state the result on existence of weak solutions.

\subsection{The model}

Consider a bounded domain $\Omega\subset\R^d$, $d=2,3$, describing 
the viscoelastic body in a reference configuration,
and a time horizon $T>0$.
The deformation of the body is described by the field $\bv\colon(0,T)\times\Omega\to\R^d$,
and the field $\psi\colon(0,T)\times\Omega\to\R$
describes the different phases of liquid absorption that result in swelling or shrinking.
The deformation field $\bv$ and the phase field $\psi$
are coupled via a multiplicative decomposition of the 
deformation gradient $\bF=\nabla\bv$ into 
an elastic part $\Fe$ and a plastic part $\Fp$.
More precisely, we let the plastic and the elastic deformation gradients be defined by
\begin{equation}\label{def-Fp_Ogden}
\Fp = \Fp(\psi) = g(\psi)\I
\qquad\text{and}\qquad
\Fe = \Fe(\bv,\psi)= \bF\Fp^{-1}=\frac{1}{g(\psi)}\nabla\bv\,,
\end{equation}
where $g$ is a function that acts as a measure of swelling,
and $\I$ is the $(d\times d)$-identity tensor.
In particular, swelling is modeled as a purely volumetric and isotropic effect.
Based on this decomposition,
the free energy density is given by
\begin{equation}
\label{eq:Wgesamt}
\Wfe(\bF,\nabla\bF,\psi,\nabla\psi)=\Wel(\Fe)+\Why(\nabla\bF)+\Wpf(\psi,\nabla\psi,\bF),
\end{equation}
where $\Fe$ and $\bF$ are related by~\eqref{def-Fp_Ogden}.
The first two terms represent elastic contributions,
where $\Wel$ corresponds to a nonlinear elastic energy
depending on the elastic deformation gradient $\Fe$,
and $\Why$ depends on $\nabla\bF$,
which leads to a hyperelastic stress in the evolution equations.
The last term $\Wpf$ is a Ginzburg--Landau free energy, 
which combines a double-well potential for $\psi$ 
with a term quadratic in $\bF^{-T}\nabla \psi$
that corresponds to a Korteweg stress in the deformed configuration.
For the detailed properties of these potentials, we refer to 
Subsection~\ref{subsec:assumptions}.

The evolution of the phase-field variable $\psi$ is governed by the system
\begin{subequations}\label{equ:pde}
\begin{align}
    \partial_t\psi &= \Delta\mu\,
    &&\text{in }[0,T]\times\Omega\,,
    \label{equ:CahnHilliard}
    \\
    \mu&=\partial_{\psi}\Wfe(\bF,\nabla\bF,\psi,\nabla\psi)-\operatorname{div}\big[\partial_{\nabla\psi}\Wfe(\bF,\nabla\bF,\psi,\nabla\psi)\big]
    &&\text{in }[0,T]\times\Omega\,,
    \label{equ:chempot}
    \\
    \nabla\mu\cdot\nu  & =\nabla\psi\cdot\nu    = 0       &&\text{on }[0,T] \times\partial\Omega\,.
    \label{equ:bdry.normal}
\end{align}
This constitutes
a Cahn--Hilliard equation,
which can be regarded as an $H^{-1}$-gradient flow \cite{garcke2000mathematical,heinemann2015existence}.
More precisely,~\eqref{equ:CahnHilliard} is a diffusion equation 
for the phase-field variable $\psi$
with constant mobility equal to $1$ and chemical potential $\mu$,
which is subject to the constitutive equation~\eqref{equ:chempot},
given as the variation of the free energy $W$ with respect to $\psi$.
Moreover,~\eqref{equ:bdry.normal} describes Neumann boundary conditions,
which ensures mass conservation.
The evolution of the deformation field $\bv$ is determined by the equations
\begin{align}
    \operatorname{div}\left[\Sigma_{\mathrm {el}} + \Sigma_{\mathrm {vi}} -\Div\Sigma_{\mathrm {hy}}\right] &= 0
    &&\text{in }[0,T]\times\Omega\,,
    \label{equ:forcebalance}
    \\
    \bv                  & =\bvD  &&\text{on }[0,T]\times\DBC\,,
    \label{equ:bdry.Dirichlet}\\ 
    [\Sigma_{\mathrm {el}}+\Sigma_{\mathrm {vi}}-\dv_\mathrm{s}\Sigma_{\mathrm {hy}} ]\nu     & = 0       &&\text{on } [0,T]\times\partial\Omega\backslash \DBC\,,
    \label{equ:bdry.Neumann}
\\
    \Sigma_{\mathrm {hy}}:(\nu\otimes\nu)     & = 0       &&\text{on } [0,T]\times\partial\Omega\,.
    \label{equ:bdry.hyperstress}
\end{align}
Here,~\eqref{equ:forcebalance} describes a
Kelvin--Voigt-type viscoelastic material at mechanical equilibrium.
In particular, we consider a quasistatical approximation 
that neglects inertia terms.
The total stress is 
composed of the first-grade elastic stress $\Sigma_{\mathrm {el}}$, the viscous stress $\Sigma_{\mathrm{vi}}$, and the elastic hyperstress $\Sigma_{\mathrm{hy}}$
given by
\begin{equation}\label{eq:stresses}
\Sigma_{\mathrm {el}}=\partial_{\bF}\Wel(\Fe)+\partial_{\bF}\Wpf(\psi,\nabla\psi,\bF),
\quad
\Sigma_{\mathrm {vi}}=\partial_{\bs{\dot{F}}}V(\bF,\dot{\bF},\psi),
\quad
\Sigma_{\mathrm {hy}}=\partial_{\nabla\bF}\Why(\nabla\bF),
\end{equation}
where $V$ is a viscous stress potential
that depends on the defomration gradient $\bF$, its time derivative $\dot{\bF}$ and the phase-field variable $\psi$.
By~\eqref{equ:bdry.Dirichlet} we prescribe the deformation at the 
Dirichlet part $\DBC\subset\partial\Omega$ of the boundary 
via a (time-dependent) function $\bvD$, 
while by~\eqref{equ:bdry.Neumann} we have natural boundary conditions 
for the total stress on the remaining part $\partial\Omega\setminus\DBC$.
Here, $\dv_\mathrm{s}$ is the surface divergence, defined as the trace
of the surface gradient $\nabla_\mathrm{s}u=(\I-\nu\otimes\nu)\nabla u$.
Equation \eqref{equ:bdry.hyperstress} is a second-order Neumann-type condition 
for the hyperstress.
For a justification of these boundary conditions,
we refer to \cite{mielke2020thermoviscoelasticity}
or~\cite[Subsect.~2.5.1]{kruvzik2019mathematical}. 
The system is completed with initial conditions
\begin{equation}
    (\psi,\bv)(0)=(\psi^0,\bv^0) \qquad \text{in }\Omega
    \label{equ:init}
\end{equation}
\end{subequations}
for the phase field and the deformation.

\subsection{Assumptions}
\label{subsec:assumptions}

For the whole article, we consider a bounded domain $\Omega\subset\R^d$, $d\in\{2,3\}$,
with Lipschitz boundary.
Let the Dirichlet part $\DBC\subset\partial\Omega$ of the boundary
have positive surface measure $\mathcal{H}^{d-1}(\DBC)>0$.
We further assume the following:
\begin{enumerate}[label=(A\arabic*)] 
    \item\label{Ass0}
    Let $p,q,\beta\in(1,\infty)$ satisfy
    \[
    p\geq 2\beta,\quad \beta >d \quad\text{and}\quad q\geq\frac{\beta d}{(\beta-d)}\,.
    \]
    \item\label{Ass1} The elastic energy density $\Wel:{\mathrm {GL}}_+(d)\to\R_{\geq 0}$ is twice continuously differentiable 
    and satisfies:
    \begin{enumerate}[label=(\roman*)] 
        \item \textit{Growth condition}: There are $\alpha, c>0$ such that for all $\bF\in{\mathrm {GL}}_+(d)$ it holds
    \begin{equation}\label{equ:prop-R_1}
        \Wel(\bA)\geq \alpha|\bA|^p+c\det(\bA)^{-q}\,.
    \end{equation}
    \item 
    \textit{Control of the stress}: 
    There is $C>0$ such that for $\bF
    \in{\mathrm {GL}}_+(d)$  
    it holds
    \begin{align}\label{equ:stress_contral_KH2}
        |\partial_{\bA}\Wel(\bA)\bA^T|\leq C\left(1+\Wel(\bA)\right)\,.
    \end{align}
    \item 
    \textit{Uniform continuity of stress}: For $\overline{\varepsilon}>0$ there exist $\delta>0$ such that for $\bA,\bB\in{\mathrm {GL}}_+(d)$ and $|\bB-\I|\leq \delta$ it holds
    \begin{align}\label{equ:uniform_continuity_stresses}
    \begin{split}
        |\partial_{\bA}\Wel(\bB\bA)(\bB\bA)^T-\partial_{\bA}\Wel(\bA)\bA^T|\leq \overline{\varepsilon} (1+\Wel(\bA))\,.
    \end{split}
    \end{align}
    \end{enumerate}

    \item\label{Ass2} The elastic hyperstress potential $\Why:\R^{\dxd\times d}\to\R_{\geq 0}$ is a strictly convex
    and continuously differentiable function
    that satisfies: 
    \begin{enumerate}[label=(\roman*)]
    \item \textit{Growth conditions}: There are constants
    $\gamma,C>0$ such that for all $\bG\in\R^{\dxd\times d}$ it holds
    \begin{align}
            \gamma|\bs{G}|^\beta-C\leq \Why(\bs{G})&\leq C(1+|\bs{G}|^\beta)\,,
            \label{equ:prop-R_2}\\
             |\partial_{\bG}\Why(\bG)|&\leq C(1+|\bG|^{\beta-1}) \,. 
            \label{equ:prop-R_22}
    \end{align} 
    \item \textit{Uniform continuity of the hyperstress}: 
        For $\overline{\varepsilon}>0$, there are $\delta>0$ and $C>0$ such that for $\bG, \bH \in\R^{\dxd\times d}$ with $|\bG-\bH|\leq\delta$ it holds
    \begin{align}\label{stress_c_reg}
        |\partial_{\bG}\Why(\bG)-\partial_{\bG}\Why(\bH)|\leq\overline{\varepsilon}C(1+|\bG|^{\beta-1})\,.
    \end{align}
    \end{enumerate}

    \item\label{Ass4} There are constants $a,b>0$ such that the phase-field energy density $\Wpf:\R\times\R^d\times\R^{d\times d}\to\R$ is given by
    \begin{align}\label{def-W_double}
        \Wpf(\psi,\nabla\psi,\bF) =  \frac{a}{4}(\psi^2-1)^2+ \frac{b}{2}|\bF^{-T}\nabla \psi|^2\,.
    \end{align}

    \item\label{Ass5} 
    The 
    viscous stress potential $V:\R^{\dxd}\times\R^{\dxd}\times\R\to\R_{\geq 0}$ 
    is given via a reduced potential $\hat{V}:\R^{\dxd}\times\R^{\dxd}\times\R\to\R_{\geq 0}$ such that 
    $V(\bF,\dot{\bF},\psi)= \hat{V}(\bF^T\bF,\dot{\bF}^T\bF + \bF^T\dot{\bF},\psi)$,
    where $\hat V$ is quadratic in the second argument with
    \begin{equation}\label{eq:Vhat.quadratic}
        \forall\psi\in\R, \ \forall \bC,\dot{\bC}\in\R^{d\times d}: \quad
        \hat{V}(\bC,\dot{\bC},\psi) = \frac{1}{2}\dot{\bC}:\bD(\bC,\psi)\dot{\bC}
    \end{equation}
    for a continuous forth-order tensor $\bD:\R^{d\times d}\times\R\to\R^{d\times d \times d \times d}$,
    and such hat there are constants $c,C>0$ with 
    \begin{align}\label{equ:fuer_korn}
        \forall\psi\in\R, \ \forall \bC,\dot{\bC}\in\R^{d\times d}: \quad
        c|\dot{\bC}|^2\leq \hat{V}(\bC,\dot{\bC},\psi)\leq C |\dot{\bC}|^2\,.
    \end{align}

    \item \label{Ass6} 
    The measure of swelling $g$ is a nondecreasing function $g\in C^1(\R)$ such that 
    \begin{align}\label{equ:alles_g-OO}
    \exists\, \underline{g},\overline{g}\in\R: \ \forall z\in\R: \quad 0<\underline{g}\leq g(z)\leq\overline{g}\,. 
    \end{align}
    
\item \label{Ass7}
The time-dependent boundary data 
$\bvD$
are given as a function
$\bvD:[0,T]\times\R^d\to\R^d$ with the regularity
\begin{equation}
\label{equ:BDC_1}
\bvD \in \mathrm C^1([0,T];\mathrm C^2(\R^d;\R^d))
\end{equation}
such that there is $C>0$ with
\begin{equation}\label{equ:BDC_2}
|\nabla_{\by}\bvD(t,\by)|+|\nabla_{\by\by}^2\bvD(t,\by)|+|\partial_t\nabla_{\by}\bvD(t,\by)|+ |\partial_t\nabla^2_{\by\by}\bvD(t,\by)|\leq C
\end{equation}
for all $\by\in\R^d$ and $t\in[0,T]$,
and with pointwise invertible spatial derivative $\nabla_{\by}\bvD(t,\by)$
such that 
\begin{equation}\label{equ:BDC_3}
    \forall (t,\by)\in [0,T]\times\R^d: \ |\nabla_{\by}\bvD(t,\by)^{-1}|\leq C.
\end{equation}
\end{enumerate}

\begin{remark}
One example that satisfies the conditions in~\ref{Ass1} and~\ref{Ass2} is given by the Ogden-type elastic energy potential $\Wel$ and the hyperstress potential $\Why$ defined as 
\begin{gather}\label{def-Wel-0}
    \Wel(\bF)  
    = \frac{\alpha}{p}|\bF|^p + \frac{c}{q}(\det(\bF))^{-q}\quad\text{and}\quad
    \Why(\bG) = \frac{\gamma}{\beta}|\bG|^{\beta} 
\end{gather}
with $\alpha,c>0$ and exponents $p,q,\beta\in (1,\infty)$ as in assumption~\ref{Ass0}. 
Additionally, this choice satisfies static frame-indifference,
see Remark~\ref{rem:frame_indiff} for more information.
A natural candidate for the swelling function $g$ could be given by
a non-decreasing function
with
\[  
    g(t)=  
    \begin{cases}
        \overline{g}\quad &\text{if }t\geq 1+\delta\,,\\
        1+a t\qquad &\text{if }t\in[-1,1]\,,\\
        \underline{g}\quad &\text{if }t\leq-1-\delta\,
    \end{cases}
\]
for some $a,\delta>0$ and $0<\underline{g}<1<\overline{g}$. 
Then
the volumetric swelling effect is linear in $\psi$
for $\psi\in[-1,1]$,
it vanishes for $\psi=0$,
and the energetically preferred states $\psi=\pm 1$ correspond to 
saturated and dry phases.
\end{remark}

\begin{remark}
By the lower bound~\eqref{equ:prop-R_1},
the energy blows up
if the determinant of the deformation gradient approaches $0$.
This ensures that local self-penetration is avoided.
Since in our analysis the deformation field $\bv$
is continuously differentiable, this further implies its local invertibility.
However,~\eqref{equ:prop-R_1} does not prevent from 
global self-penetration of the material.
\end{remark}

\begin{remark}\label{rem:dAdFWel}
In~\eqref{equ:stress_contral_KH2} and~\eqref{equ:uniform_continuity_stresses}
we formulate assumptions on the classical derivative $\partial_{\bA}\Wel$
of the elastic potential,
which can also be regarded as the derivative with respect to the elastic part of the deformation gradient,
which will occur as the argument of $\Wel$.
However, in our model the elastic stress $\Sigma_{\mathrm{el}}$ is determined by $\partial_{\bF}\Wel$,
that is, the derivative with respect to the full deformation gradient.
In virtue of~\eqref{def-Fp_Ogden}, these terms are related by
\begin{equation}
\label{eq:Wder.F}
\partial_{\bF}\Wel(\Fe)=\partial_{\bA}\Wel (\Fe) \Fp^{-T} = \frac{1}{g(\psi)}\partial_{\bA}\Wel (\Fe)
\end{equation}
Similarly, we observe
\[
    \partial_{\psi}\Wel(\Fe)
    = \partial_{\bA}\Wel(\Fe)\bF^T\partial_{\psi}\left(\tfrac{1}{g(\psi)}\right)
    = -\partial_{\bA}\Wel(\Fe)\bF^T\tfrac{g'(\psi)}{g(\psi)^2}\,.
\]
\end{remark}

\begin{remark}
The upper bounds and continuity properties of the 
derivatives of the elastic stress and hyperstress potentials
in~\eqref{equ:stress_contral_KH2},~\eqref{equ:uniform_continuity_stresses},
\eqref{equ:prop-R_22} and~\eqref{stress_c_reg}
will be used to identify minimizers of the time-discretized problem
as solutions of associated Euler--Lagrange equations.
In the time-continuous limit, this leads to a weak formulation
of the mechanical force balance~\eqref{equ:forcebalance}
in a Lagrangian frame, see~\eqref{equ:elast} below.
Observe that the condition~\eqref{equ:stress_contral_KH2}
was used by Ball~\cite{ball2002some}
for a similar purpose
in the framework without a second-grade elasticity term.
However, the admissible test functions used in the weak formulation in~\cite{ball2002some} 
depend on the weak solution of the problem.
Instead, the formulation in a Lagrangian frame
and with test functions independent of the solution,
which we use here,
seems favorable
from the perspective of numerical implementation.
\end{remark}

\begin{remark}
The phase-field energy density $\Wpf$ in~\eqref{def-W_double} 
is composed of a double-well potential and a gradient term
associated to a Korteweg stress.
The first term favors pure phases $\psi=\pm 1$, 
while the second one penalizes transitions between these phases
thus modeling capillarity effects. 
In order to take into account the phase-field gradient 
in the deformed configuration, 
we consider here $\bF^{-T}\nabla\psi$ 
instead of $\nabla\psi$.
Choosing $a=o(\varepsilon^{-1})$ and $b=o(\varepsilon)$,
the limit $\varepsilon\to 0$ formally leads to a sharp-interface model
with surface tension.
\end{remark}

\begin{remark}\label{remark:viscous.stress}
Instead of~\ref{Ass5},
we could choose 
the viscous potential $V$ quadratic in $\dot\bF$
in order to simplify the analysis.
However, for a physically reasonable model,
$V$ has to satisfy 
\textit{dynamic frame-indifference},
which means that for all $\bF,\,\dot{\bF}\in\R^{d\times d}$ and all smooth functions $t\mapsto\bs{R}(t)\in\mathrm{SO}(d)$ it holds
    \[
        V(\bs{R}\bF,\partial_t\bs{R}\bF+\bs{R}\dot{\bF},\psi)
        = V(\bF,\dot{\bF},\psi),
    \]
see also~\cite[Sect.~9]{kruvzik2019mathematical}.
This is equivalent to the existence of a reduced potential $\hat V$ in terms of the right Cauchy--Green tensor 
$\bC=\bF^T\bF$ and its time derivative $\dot \bC= \dot \bF^T\bF+\bF^T\dot\bF$
as we assume in~\ref{Ass5}.
In this regard, the quadratic structure of $\hat V$ in~\ref{eq:Vhat.quadratic},
which leads to a viscous stress $\Sigma_{\mathrm{ vi}}$ 
linear in $\dot\bC$,
is the simplest choice that is physically reasonable.
Observe that we allow the forth-order tensor $\bD$ to depend on $\bC$, $\psi$
in an arbitrary (smooth) nonlinear fashion,
but even if $\bD$ is constant, 
the viscous potential $V$ explicitly depends on $\bF$.
\end{remark}

\begin{remark}
\label{rem:frame_indiff}
To obtain a physically reasonable model, 
the elastic stress potential $\Wel$
    and the hyperstress potential $\Why$
    has to satisfy \emph{static frame-indifference}: 
    For all rotation matrices $\bs{R}\in\mathrm{SO}(d)$ it holds
    \begin{align*}
        \Wel(\bs{R}\bF,\psi) = \Wel(\bF,\psi)
        \quad\text{and}\quad
        \Why(\bs{R}\bG) = \Why(\bG) \,
    \end{align*}
    for all $\bF\in\R^{d\times d}$ and $\bG\in\R^{d\times d\times d}$,
    where $(\bs{R}\bG)_{jk\ell}:=\sum_{m=1}^d \bs{R}_{jm}\bG_{mk\ell}$.
    Note that one can also formulate the static frame-indifference of the elastic potentials
    and the dynamic frame-indifference of the viscous potential 
    (see Remark~\ref{remark:viscous.stress})
    in terms of the associated stresses $\Sigma_{\mathrm{el}}$, $\Sigma_{\mathrm{hy}}$ and $\Sigma_{\mathrm{ vi}}$, which are given in~\eqref{eq:stresses}, in the following way:
    For all rotation matrices $\bs{R}\in\mathrm{SO}(d)$ it holds
    \[
        \Sigma_{\mathrm{el}}(\bs{R}\bF,\psi) = \bs{R}\Sigma_{\mathrm{el}}(\bF,\psi)
        \quad\text{and}\quad
        \Sigma_{\mathrm{hy}}(\bs{R}\bG) = \bs{R}\Sigma_{\mathrm{hy}}(\bG) \,,
    \]
    and for all smooth functions 
    $t\mapsto\bs{R}(t)\in\mathrm{SO}(d)$ it holds
    \[
        \Sigma_{\mathrm {vi}}(\bs{R}\bF,\partial_t\bs{R}\bF+\bs{R}\dot{\bF},\psi)
        = \bs{R}\Sigma_{\mathrm {vi}}(\bF,\dot{\bF},\psi)
    \]
    for all $\bF, \dot{\bF}\in\R^{d\times d}$ and $\bG\in\R^{d\times d\times d}$.
\end{remark}

\subsection{Functional setup}

In view of the time-dependent Dirichlet conditions~\eqref{equ:bdry.Dirichlet}
and the growth condition~\eqref{equ:prop-R_1}
for the elastic energy,
the set of admissible deformations at time $t\in[0,T]$ is given by 
\begin{equation}\label{def-Ut}
    \bfU(t) =\{\bw\in W^{2,\beta}({\Omega};\R^d)\,\vert\,\bw=\bvD(t) \text{ on } \DBC, \, \det(\nabla \bw)>0 \text{ in }\Omega\}\,,
\end{equation}
which is a closed subspace of $W^{2,\beta}(\Omega;\R^d)$.
To avoid this time-dependent function space, we introduce
\begin{equation}\label{def-Y}
    \bfY = \{\by\in W^{2,\beta}({\Omega}; \R^d)\,\vert\,\by = \id\text{ on }\DBC,\, \det(\nabla \by)>0 \text{ in }\Omega\}\,,
\end{equation}
which is again a closed subspace of $W^{2,\beta}(\Omega;\R^d)$.
Under the assumption~\ref{Ass7}, $\bvD$ induces a bijective mapping between $\bfY$ and $\bfU(t)$ by
\begin{equation}\label{equ:DBC_with_y}
    \by\mapsto\bv(t,\cdot):=\bvD(t,\by(\cdot)),
\end{equation}
see Lemma~\ref{lemma:estvy} below.
In particular, $\bv(t,x) = \bvD(t,x)$ for $x\in\DBC$ for all $t\in[0,T]$,
and $\nabla\bv(t,\cdot)$ has positive determinant if and only if $\nabla\by$ does.

The set of admissible phase fields $\psi$ is given by
\[
    \bfZ = H^1({\Omega})\,. 
\]
Observe that from~\eqref{equ:CahnHilliard} and~\eqref{equ:bdry.normal}, we can conclude conservation of mass, precisely,
\begin{align*}
    \forall t\in[0,T]: \ \intO \left(\psi(t) - \psi^0\right)\dx  = 0 \,.
\end{align*}
In particular, $\partial_t\psi$ has vanishing mean value,
and we can formulate the Cahn--Hilliard equation~\eqref{equ:CahnHilliard} as 
an $H^{-1}$-gradient flow. To this end, we introduce the mean-free function space $V_0$ and its dual space by 
\begin{subequations}\label{def-V0_H-1GF}
\begin{equation}\label{equ:_V0_Neo}
    V_0  = \big \{\phi\in H^1({\Omega})\mid\intO \phi\dx = 0\big\}\quad\text{and}\quad
    \Tilde{V}_0  = \big\{\phi\in (H^1({\Omega}))^{*}\mid \langle\phi,{1}\rangle_{(H^1)^{*}\times H^1} = 0\big\}\,.
\end{equation}
This allows us to define the inverse operator $(-\Delta)^{-1}:\Tilde{V}_0 \to V_0$ of 
\begin{equation}\label{equ:Inverse_laplace}
    -\Delta:V_0\to\Tilde{V}_0\,,\quad\phi\mapsto(\nabla\phi,\nabla\cdot)_{L^2(\Omega)}. 
\end{equation}
The space $\Tilde{V}_0$ is endowed with the scalar product 
\begin{equation}\label{equ:NEO_skp}
    \langle \phi_1,\phi_2\rangle_{\Tilde{V}_0}=(\nabla(-\Delta)^{-1}\phi_1,\nabla(-\Delta)^{-1}\phi_2)_{L^2(\Omega)}\quad\text{for }\phi_1,\,\phi_2\in \Tilde{V}_0\,
\end{equation}
\end{subequations}
and the corresponding norm.
Note that the definition in \eqref{equ:NEO_skp} is suitable for constant Cahn--Hilliard mobility as assumed for our system, but it has to be adapted for other mobilities.
Moreover, each function $f\in V_0$ can be interpreted as an element of $\Tilde{V}_0$ via the mapping
\begin{equation*}
    \xi \mapsto \intO f\cdot\xi\dx  \quad\text{for }\xi\in V_0\,, 
\end{equation*}
such that $\langle\cdot,\cdot\rangle_{\Tilde{V}_0}$ is also defined for elements in $V_0$.

The free energy density given by~\eqref{eq:Wgesamt}
gives rise to the total free energy
\begin{equation}\label{def-Fg}
\begin{aligned}
    &\fe:[0,T]\times\bfY\times\bfZ\to\R\cup\{\infty\},
    \\
    &\fe (t, \by, \psi) 
    =  \intO\big[\Wel(\bF(t,\by)\Fp^{-1}(\psi)) +\Wpf(\psi,\nabla\psi,\bF(t,\by)) + \Why(\nabla\bF(t,\by))\big]\dx\,.
\end{aligned}
\end{equation}
Further, the viscous dissipation $\calV$ is defined via 
the viscous dissipation potential $V$
and takes the form
\begin{equation}\label{eq:viscousdissipation}
    \begin{aligned}
    &\calV\colon[0,T]\times\bfY\times H^1(\Omega;\R^d)\times\bfZ \to \R,
    \\
    &\calV(t,\by_1,\by_2,\psi) = \intO V(\bF(t,\by_1),\bF(t,\by_2),\psi)\dx\,.
\end{aligned}
\end{equation}
Usually, we will consider the case $\by_2=\partial_t\by_1$, 
but other choices will be necessary in the time-discrete setting.

\subsection{Existence of weak solutions}

We can now define the notion of weak solutions to~\eqref{equ:pde}.
For given initial data $(\by^0,\psi^0)\in\bfY\times\bfZ$,
a pair $(\by,\psi):[0,T]\to\bfY\times\bfZ$ 
is called a \textit{weak solution} 
to \eqref{equ:pde} if
\begin{enumerate}[label=(\roman*)]
    \item 
    $\byg\in L^{\infty}(0,T;W^{2,\beta}(\Omega;\R^{d}))\cap H^1(0,T;H^1(\Omega;\R^d))$,
    \item 
    $\psig\in L^{\infty}(0,T;H^1(\Omega))\cap H^1(0,T;H^1(\Omega)^*)$, 
    \item
    $(\byg(0),\psig(0))=(\by^0,\psi^0)$,
    \item  there exists $\mu\in L^2(0,T;H^1(\Omega))$ such that for all $ {\zeta\in L^2(0,T;H^1({\Omega}))}$ it holds
    \begin{align}
        &\int_0^T\langle\partial_t\psig,\zeta \rangle_{(H^1)^*\times H^1} \ds
         =  -\int_{\TOmega}   \nabla \mu \cdot \nabla\zeta \dxds\,,\label{equ:time}
         \\
         &\int_{\TOmega}  \mug \,\zeta \dxds 
         = \int_{\TOmega}\partial_{\psi} W(\bFg,\nabla\bFg,\psig,\nabla\psig) \zeta 
         +\partial_{\nabla\psi} W(\bFg,\nabla\bFg,\psig,\nabla\psig)\cdot\nabla\zeta \dxds\label{equ:chem}\,,
    \end{align}
    \item 
    for all $\bw\in L^2(0,T;W_{\DBC}^{2,\beta}({\Omega};\R^d))$ it holds
    \begin{equation}\label{equ:elast}
        \begin{aligned}
         0 = \int_{\TOmega}\big[
         &\partial_{\bA}\Wel\left(\bF_{e}\right):\nabla\bw\Fp^{-1}
        +\partial_{\bF} \Wpf(\psi,\nabla\psi,\bF)
        \\
        &\qquad
        +\partial_{\bs{\dot{F}}}V(t,\bFg,\partial_t\bFg):\nabla\bw 
        +\partial_{\bG}\Why(\nabla\bFg)\tdots\nabla^2\bw\big]\dxds\,,
        \end{aligned}
    \end{equation}
\end{enumerate}
where $\bF(t,x):=\nabla\bvD(t,\by(t,x))$, $\Fp(t,x):=g(\psig(t,x))\I$
and $\Fe:=\bF\Fp^{-1}$.

Here the initial data $\by^0$ and $\bv^0$ are related by $\bv^0=\bvD(0,\by^0)$.
In this sense, a weak solution $(\by,\psi)$ satisfies the initial conditions~\eqref{equ:init}.

The weak formulation~\eqref{equ:time},~\eqref{equ:chem} and~\eqref{equ:elast}
is in a compact form.
Due to~\eqref{def-Fp_Ogden} and~\eqref{def-W_double}, 
more detailed expressions for the derivatives in~\eqref{equ:chem}
are given by
\begin{align}
\partial_\psi W(\bF,\nabla\bF,\psi,\nabla\psi)
&=-\partial_{\bA} \Wel \left(\frac{1}{g(\psi)}\bF\right):\bF \frac{g'(\psi)}{g(\psi)^2}
+a(\psi^3-\psi),
\label{eq:Wder.psi}
\\
\partial_{\nabla\psi} W(\bF,\nabla\bF,\psi,\nabla\psi)
&=\partial_{\nabla\psi}\Wpf(\psi,\nabla\psi,\bF)=b\bF^{-1}\bF^{-T}\nabla\psi.
\label{eq:Wder.dpsi}
\end{align}
In~\eqref{equ:elast} we replaced the derivative of $\Wel$ 
with respect to $\bF=\nabla\bv$,
which appears in the elastic stress $\Sigma_{\mathrm{el}}=\partial_{\bF}\Wfe$,
with a term involving the classical derivative of $\Wel$,
see also Remark~\ref{rem:dAdFWel}.
The corresponding derivative of $\Wpf$ is given by
\begin{equation}
\partial_{\bF}\Wpf (\psi,\nabla\psi,\bF) :\bG= -b (\bF^{-1}\bF^{-T}\nabla\psi) : (\bG^T\bF^{-T}\nabla\psi),
\qquad\bG\in\R^{d\times d}.
\label{eq:Wpfder.F}
\end{equation}

Clearly, \eqref{equ:time} and~\eqref{equ:chem}
are weak formulations of~\eqref{equ:CahnHilliard}
and~\eqref{equ:chempot}
in virtue of the boundary conditions in~\eqref{equ:bdry.normal}.
The corresponding statement can also be shown for~\eqref{equ:elast},
see~\cite{mielke2020thermoviscoelasticity} for details.

As the main result of this article,
we show existence of
weak solutions to~\eqref{equ:pde}
globally in time
that satisfy an energy-dissipation estimate.
Here we require a smallness condition that involves 
the boundary velocity $\bvD$ and the 
constants $\alpha$ and $\gamma$
occurring in the growth conditions for the elastic potential
and the hyperstress potential from~\ref{Ass1} and \ref{Ass2}.

\begin{theorem}[Existence of weak solutions]\label{Thm:main_result_fast}
Let $\Omega\subset\R^d$, $d\in\{2,3\}$, be a bounded Lipschitz domain,
let $\mathcal{H}^{d-1}(\DBC)>0$,
and let
the assumptions \ref{Ass0}--\ref{Ass7} hold true.
Then there is 
$\delta_0>0$ such that if
\begin{equation}\label{eq:smallness}
\alpha^{-1}\gamma\norm{\nabla^2\bvD}_{L^\infty((0,T)\times\Omega)}^\beta< \delta_0,
\end{equation}
then there exists a weak solution $(\by,\psi)$ to system \eqref{equ:pde} such that
the energy-dissipation estimate
    \begin{equation}\label{equ:energy_O}
    \begin{aligned}
        \fe (t,\byg(t),\psig(t)) + \int_{0}^{t}\int_{\Omega}
        \frac{1}{2}|\nabla\mu|^2 &\dx \ds
        +\int_{0}^{t} \calV (s,\byg(s),\partial_t\byg(s),\psi(s)) \ds 
        \\
        & \leq \fe (0, \by^0,\psi^0) + 
        \int_{0}^{t}\partial_t\fe (s,\by(s),\psi(s))\ds\,
    \end{aligned}
    \end{equation}
holds for a.a.~$t\in[0,T]$.
\end{theorem}

The proof of Theorem~\ref{Thm:main_result_fast}
will be given in Section~\ref{sec:finalproof}.

If $(\by,\psi)$ is a weak solution as in Theorem~\ref{Thm:main_result_fast},
then we can use the transformation~\eqref{equ:DBC_with_y},
to obtain the actual deformation 
$\bv(t,x):=\bvD(t,\by(x))$
from the auxiliary variable $\by$.
Then $(\bv,\psi)$ has the same regularity as $(\by,\psi)$ 
and satisfies~\eqref{equ:time},~\eqref{equ:chem} and~\eqref{equ:elast} with $\bF:=\nabla\bv$.
Thus, $(\bv,\psi)$ can also be considered a weak solution to~\eqref{equ:pde}.
Nevertheless, we chose the formulation of Theorem~\ref{Thm:main_result_fast}
in terms of the auxiliary variable $\by$
since it
enables us to identify the term $\partial_t\fe$
in the energy-dissipation inequality~\eqref{equ:energy_O}
as the power of external forces
due to the time-dependent boundary conditions.
Moreover, \eqref{equ:energy_O} reveals that the energetic principles behind the 
time evolution of the system:
It combines energy storage via the free energy $\fe$ 
with energy dissipation due to two effects,
one arising from the chemical potential $\mu$,
and the other due to the viscosity potential $V$
via~\eqref{eq:viscousdissipation}.

Observe that the smallness condition 
\eqref{eq:smallness} is trivially satisfied if $\bvD(t,\cdot)$ is affine linear for every $t\in[0,T]$
since then $\nabla^2\bvD=0$.
In particular, the prescription of time-dependent dilatation or shearing
of the boundary part $\DBC$ is admissible without size restrictions.
If $\nabla^2\bvD=0$ is not satisfied,
the second-order gradient $\nabla^2\bv$ of the deformation 
also contains first-order derivatives of the auxiliary variable $\by$,
see~\eqref{equ:grad_v_NH} below.
These first-order terms thus appear in the hyperstress potential $\Why$
and 
the smallness condition ensures their compensation by the elastic potential $\Wel$
due to the growth condition in assumption~\ref{Ass1},
so that the free energy $\fe$ is coercive
in terms of the variables $(\by,\psi)$,
as shown in Lemma~\ref{lemma:Coercive_LowerSemi} below.

\section{Preliminary results}
\label{sec:preliminaries}

In this section, we first study how convergence properties 
can be transferred
from the auxiliary variable $\by$
or the phase-field $\psi$  to convergence of associated functions.
Afterwards, we 
derive several properties of the energy functional $\fe$.

\subsection{Convergence properties}

By the construction in \eqref{equ:DBC_with_y}, the deformation gradient has a multiplicative structure when expressed in the auxiliary variable $\by$, 
namely
\begin{align}\label{equ:grad_v_NH}
    \bF(t,x)=\nabla \bv(t,x) = \nabla_{\by}\bvD(t,\by(t,x))\nabla \by(t,x)\quad\text{for }(t,x)\in[0,T]\times\Omega\,,
\end{align}
where $\nabla_{\by}\bvD$ denotes the derivative of $\bvD$ with respect to the second argument $\by$.
The second-order gradient of $\bv:\Omega\to\R^d$, which we denote by $\nabla\bF=\nabla^2\bv$, is a third-order tensor
with components 
\begin{align}\label{equ:componentwise-zweite}
\begin{split}
    \left(\nabla\bF\right)_{ijk}
    =\left(\nabla^2\bv\right)_{ijk}
    =\partial_j\partial_k\bv_i
    & =  {(\nabla\by_j)}^T{(\nabla_{\by\by}^2\bv_{\D,i})}{(\nabla\by_k)}
    + (\nabla_{\by}\bv_{\D,i})^T{\partial_j\partial_k\by}\,
\end{split}
\end{align}
for $i,j,k\in\{1,\dots,d\}$.

In the following lemma, we compare Lebesgue and Sobolev norms 
of $\by$ and $\bv=\bvD(t,\by)$.
In particular, we conclude that the relation \eqref{equ:DBC_with_y} induces a homeomorphism between
$\bfY$ and $\bfU(t)$
defined in~\eqref{def-Ut} and~\eqref{def-Y},
and that strong convergence transfers from one space to the other.

\begin{lemma}\label{lemma:estvy}
Let $\bvD$ be as in \ref{Ass7}.
Then $\bvD(t.\cdot)$ is bijective for all $t\in[0,T]$,
and there is a constant $C>0$, independent of $t$, such that
for all  $\by\in\bfY$ and $\bv(t,x)=\bvD(t,\by(x))$
and all $p\in[1,\infty)$,
we have
\begin{align}
    \norm{\bv}_{W^{1,p}(\Omega)}
    &\leq C(1+\norm{\by}_{W^{1,p}(\Omega)}),
    \label{est:vy}
    \\
    \norm{\by}_{W^{1,p}(\Omega)}
    &\leq C(1+\norm{\bv}_{W^{1,p}(\Omega)}),
    \label{est:yv}
    \\
    \norm{\nabla^2\bv}_{L^p(\Omega)}
    &\leq C(\norm{\nabla\by}_{L^{2p}(\Omega)}^2
    +\norm{\nabla^2\by}_{L^p(\Omega)}),
    \label{est:vyder2}
    \\
    \norm{\nabla^2\by}_{L^p(\Omega)}
    &\leq C(\norm{\nabla\bv}_{L^{2p}(\Omega)}^2
    +\norm{\nabla^2\bv}_{L^p(\Omega)}).
    \label{est:yvder2}
\end{align}
Moreover, $\by\mapsto\bvD(t,\by)$
defines a homeomorphic mapping from $\bfY$ to $\bfU(t)$
if $\beta\in(1,\infty)$ with $\beta\geq d/2$.
\end{lemma}

\begin{proof}
Bijectivity of $\bvD(t,\cdot)$ follows directly from
the invertibility of $\nabla \bvD(t,\cdot)$ in all points
and Hadamard's inverse function theorem.
Since $\nabla \bvD(t,\cdot)$ and its matrix inverse are bounded,
from~\eqref{equ:grad_v_NH} we directly obtain
\begin{equation}\label{est:vyder}
    \norm{\nabla\bv}_{L^p(\Omega)}
    \leq C\norm{\nabla\by}_{L^p(\Omega)},
    \qquad
     \norm{\nabla\by}_{L^p(\Omega)}
    \leq C\norm{\nabla\bv}_{L^p(\Omega)}.
\end{equation}
Moreover, we can use Poincar\'e's inequality 
to conclude
\[
\begin{aligned}
\norm{\bv}_{L^p(\Omega)}
&\leq \norm{\bv-\bvD}_{L^p(\Omega)}
+\norm{\bvD}_{L^p(\Omega)}
\\
&\leq C\norm{\nabla(\bv-\bvD)}_{L^p(\Omega)}
+\norm{\bvD}_{L^p(\Omega)}
\leq C(1+\norm{\nabla\bv}_{L^p(\Omega)})
\end{aligned}
\]
and
\[
\begin{aligned}
\norm{\by}_{L^p(\Omega)}
&\leq \norm{\by-\id}_{L^p(\Omega)}
+\norm{\id}_{L^p(\Omega)}
\\
&\leq C\norm{\nabla(\by-\id)}_{L^p(\Omega)}
+\norm{\id}_{L^p(\Omega)}
\leq C(1+\norm{\nabla\by}_{L^p(\Omega)}).
\end{aligned}
\]
Combining these estimates yields~\eqref{est:vy} and~\eqref{est:yv}.
Similarly,~\eqref{equ:componentwise-zweite} and~\eqref{equ:BDC_2}
directly yield~\eqref{est:vyder2}.
By multiplying with the inverse matrix $(\nabla\bvD)^{-1}$
we further deduce from~\eqref{equ:componentwise-zweite} that
\begin{equation}
\partial_j\partial_k\by
= (\nabla\bvD)^{-1}\big(\partial_j\partial_k \bv - (\nabla \by_j\otimes\nabla\by_k):\nabla^2\bvD \big),
\label{eq:grad2y}
\end{equation}
which allows to conclude~\eqref{est:yvder2} from~\eqref{equ:BDC_2}
and the first-order estimate~\eqref{est:vyder}.

If $\beta\geq d/2$, we have $W^{1,\beta}(\Omega)\hookrightarrow L^{2\beta}(\Omega)$.
Hence, the previous estimates show that $\by\in W^{2,\beta}(\Omega)$
if and only if $\bv(t)=\bvD(t,\by)\in W^{2,\beta}(\Omega)$.
Moreover, since
\[
\det(\nabla\bv(t))=\det(\nabla_{\by}\bvD(t,\by))\det(\nabla\by),
\]
we see that $\det(\nabla\bv(t))>0$ if and only if $\det(\nabla\by)>0$.
Therefore, we obtain the bijectivity of the mapping
$\bfY\to\bfU(t)$, $\by\mapsto\bvD(t,\by)$.
To show its continuity, consider a convergent sequence $(\by_k)_k\subset\bfY$
with (strong) limit $\by\in\bfY$,
and let $\bv_k(t)=\bvD(t,\by_k)$ and $\bv(t)=\bvD(t,\by)$.
We can choose a subsequence $(\by_{k_m})_m$ converging a.e.,
and using the continuity properties of $\bvD$,
we obtain $\bv_{k_m}(t,x)\to\bv(t,x)$, 
$\nabla\bv_{k_m}(t,x)\to\nabla\bv(t,x)$
and $\nabla^2\bv_{k_m}(t,x)\to\nabla^2\bv(t,x)$
a.e.~as $m\to\infty$
from~\eqref{equ:grad_v_NH} and~\eqref{equ:componentwise-zweite}.
In virtue of the Sobolev embeddings $W^{1,\beta}(\Omega)\hookrightarrow L^{2\beta}(\Omega)$,
we can use the previous estimates and 
Pratt's theorem (see \cite[Ch.\,6, Satz 5.1]{elstrodt2018mass} for instance)
to conclude that $\bv_{k_m}(t)\to\bv(t)$ in $\bfU(t)$ as $m\to\infty$.
Since the limit is the same for any choice of subsequence, 
we obtain convergence of the original sequence $(\bv_{k})_k$.
To show continuity of the inverse mapping 
from $\bfU(t)$ to $\bfY$, 
we can proceed analogously and make use of~\eqref{eq:grad2y}
instead of~\eqref{equ:componentwise-zweite}.
This completes the proof.
\end{proof}

The following result is fundamental to transfer weak convergence properties of the auxiliary variable $\by\in\bfY 
$ 
to the deformation $\bv=\bvD(t,\by)\in\bfU(t)$.

\begin{lemma}\label{lemma:conv_gradient_strong}
Let assumption~\ref{Ass7} hold true.
Let $\by\in\bfY$, and let ${(\by_k)}_{k\in\N}\subset\bfY$ be weakly convergent with $\by_k\rightharpoonup\by$ in $W^{2,\beta}(\Omega;\R^d)$ 
for $\beta>d$.
Then, for all $t\in[0,T]$   it follows 
\begin{equation}\label{equ:conv2}
    \bv_k(t)=\bvD(t,\by_k)\rightharpoonup \bvD(t,\by) = \bv(t) \quad\text{in}\quad\bfU(t)\,.
\end{equation}
\end{lemma}

\begin{proof} 
Due to the compact embedding $W^{2,{\beta}}(\Omega;\R^d)\Subset C^1(\overline{\Omega};\R^d)$, by Rellich's theorem, we have
\begin{align}\label{uni-conv}
    \by_k\to\by\,, \nabla\by_k\to\nabla\by \quad \text{ uniformly in } \overline{\Omega} \text{ as } k\to\infty\,.
\end{align}
Since $\bvD(t,\cdot)$ and $\nabla_{\by}\bvD(t,\cdot)$ are continuous by \ref{Ass7},
we conclude the convergence
\begin{equation}\label{equ:strake_conv}
\begin{alignedat}{2}
    \bv_k(t) =\bvD(t,\by_k)
    &\to\bvD(t,\by)=\bv(t) \quad && \text{in }C(\overline{\Omega};\R^d)\,,\\
    \nabla \bv_k(t) =\nabla_{\by}\bvD(t,\by_k)\nabla \by_k
    &\to\nabla_{\by}\bvD(t,\by)\nabla \by= \nabla \bv(t)\quad && \text{in }\,C(\overline{\Omega};\R^{\dxd})\,
\end{alignedat} 
\end{equation}
for all $t\in[0,T]$.
Moreover, for $i,j,\ell\in\{1,\dots,d\}$, we have
\begin{align*}
    {\left(\nabla^2\bv_{k}(t,x)\right)_{ij\ell}} 
    & =   {(\nabla\by_k)}_{j}^T {(\nabla_{\by\by}^2\bv_{\D,i}(t,\by_k))}{(\nabla\by_k)}_{\ell} + {(\nabla_{\by}\bv_{\D,i}(t,\by_k))_{\ell}}^T{(\nabla_{j\ell}^2\by_k)}\,.
\end{align*} 
The first term converges strongly in $C(\overline{\Omega})$ by the regularity assumptions on $\bvD$ and the uniform convergence 
from~\eqref{uni-conv},
and weak convergence of the second term in $L^{\beta}(\Omega)$
follows from \eqref{equ:strake_conv} and the weak convergence $\nabla^2\by_k\rightharpoonup\nabla^2\by$. 
This shows the weak convergence 
\begin{align}\label{equ:schwache_konv}
    \nabla^2\bv_k(t)\rightharpoonup\nabla^2\bv(t)\quad\text{in }L^{\beta}(\Omega;\R^{\dxd  \times d})\,
\end{align}
for $t\in[0,T]$.
Summarizing \eqref{equ:strake_conv} and \eqref{equ:schwache_konv} yields \eqref{equ:conv2}. 
\end{proof}

Next we study the convergence of the plasticity term $g(\psi)$
when we have weak convergence in the phase-field variable $\psi$.

\begin{lemma}\label{lemma:g_conv_NEO}
Let $(\psi_k)_{k\in\N}\subset H^1(\Omega)$ and $\psi\in H^1(\Omega)$ such that $\psi_k\rightharpoonup\psi$ in $H^1(\Omega)$
as $k\to\infty$.
Let $g$ satisfy assumption~\ref{Ass6},
and define $g_k(x):=g(\psi_k(x))$ and $g_\infty(x):=g(\psi(x))$.
The following properties are valid:
\begin{enumerate}[label=(\roman*)]
    \item \label{item:null}  For all $q\in[1,6)$ it holds 
    \begin{equation}\label{equ:strong_gk}
    g_k\to g_\infty\quad\text{ and }\quad\frac{1}{g_k}\to\frac{1}{g_\infty}\quad\text{in }L^q(\Omega) \quad \text{ as } k\to\infty
    \end{equation} 
    \item \label{item:eins} Let $p\in[1,\infty)$ and let $(\bG_k)_{k\in\N}\subset L^{p}(\Omega;\R^{\dxd})$ and $\bG\in L^{p}(\Omega;\R^{\dxd})$ 
    such that $\bG_k\to\bG$ in $L^{p}(\Omega;\R^{\dxd})$. Then 
    \begin{equation}\label{equ:strong-conv-g}
    { \frac{1}{g_k} }\bG_k\to\frac{1}{g_\infty} \bG\quad\text{ in }
    L^{p}(\Omega;\R^{\dxd})\,.        
    \end{equation}  
    \item \label{item:zwei} Let $p\in(1,\infty)$ and let $(\bF_k)_{k\in\N}\subset L^{p}(\Omega;\R^{\dxd})$ and $\bF\in L^{p}(\Omega;\R^{\dxd})$ such that $\bF_k\rightharpoonup\bF$ in $L^{p}(\Omega;\R^{\dxd})$. Then 
    \begin{equation}\label{equ:weak-conv-g}
    \frac{1}{g_k}\bF_k\rightharpoonup\frac{1}{g_\infty}\bF\quad\text{ in }
    L^{p}(\Omega;\R^{\dxd})\,.        
    \end{equation} 
\end{enumerate}  
\end{lemma}

\begin{proof}
From Rellich's compactness theorem, we first conclude that $\psi_k\to\psi$ in $L^q(\Omega)$ for $q\in[1,6)$, with $d\in \{2,3\}$. 
Since $|g'|$ is bounded due to assumption~\ref{Ass6}, 
we can apply the mean-value theorem to conclude that 
\[
    \|g(\psi_k)-g(\psi)\|_{L^q(\Omega)}\leq\sup_{z\in\R}|g'(z)|\|\psi_k-\psi\|_{L^q(\Omega)}\to 0\quad\text{as }k\to\infty\,.
\]
Moreover, using that  $g$ is bounded from below by \eqref{equ:alles_g-OO}, we obtain
\begin{align*}
    \|\frac{1}{g(\psi_k)}-\frac{1}{g(\psi)}\|_{L^q(\Omega)}
    \leq \frac{1}{\underline{g}^2}\|g(\psi_k)-g(\psi)\|_{L^q(\Omega)}\to 0\quad\text{as }k\to\infty\,,
\end{align*}
which shows \ref{item:null}.

Concerning \ref{item:eins}, 
we estimate
\[
\norm{g_k^{-1}\bG_k-g_\infty^{-1}\bG}_{L^p(\Omega)}
\leq \norm{g_k^{-1}(\bG_k-\bG)}_{L^p(\Omega)}
+ \norm{(g_k^{-1}-g_\infty^{-1})\bG}_{L^p(\Omega)}
\to 0 \quad\text{as }k\to\infty\,,
\]
where the first term converges due to the boundedness of $g$ from below,
and the second term converges by Lebesgue's theorem.
We thus conclude \eqref{equ:strong-conv-g}. 

To show \ref{item:zwei}, let $\bH\in L^{p'}(\Omega;\R^{\dxd})$ be arbitrary with $p'=\nicefrac{p}{(p-1)}$. 
From \eqref{equ:strong_gk} we conclude 
that $g^{-1}_k\bH\to g_\infty^{-1}\bH$ in $L^{p'}(\Omega)$.
We thus obtain 
\begin{align*}
   \intO\left(\frac{1}{g_k}\bF_k{-}\frac{1}{g_\infty}\bF\right){:}\bH\dx
    & = \intO\left(\bF_k {-}\bF\right){:}\frac{1}{g_\infty}\bH\dx +  \intO\bF_k{:}\left(\frac{1}{g_k}\bH{-} \frac{1}{g_\infty}\bH\right)\dx\to 0
\end{align*}
as $k\to\infty$. 
This implies the weak convergence asserted in \eqref{equ:weak-conv-g} and concludes the proof.
\end{proof}

\subsection{Properties of the energy functional}
We next prove additional properties of the energy functional $\fe$
defined in \eqref{def-Fg}.
At first, we show weak sequential compactness of the sublevel sets
provided that the smallness condition~\eqref{eq:smallness} is satisfied.

\begin{lemma}[Coercivity and  weak lower semiconituity of $\fe$]
\label{lemma:Coercive_LowerSemi} 
Let the assumptions~\ref{Ass0}--\ref{Ass4}, \ref{Ass6} and~\ref{Ass7}
hold true.
Then there exists $\delta_0>0$ such that if 
\[
\alpha^{-1}\gamma\norm{\nabla^2\bvD}_{L^\infty((0,T)\times\Omega)}^\beta<\delta_0,
\]
then the following properties are valid:
\begin{enumerate}[label=(\roman*)]
\item For every $t\in[0,T]$,
the functional $\fe (t,\cdot,\cdot)$ is coercive on $\bfY \times\bfZ$
in such a way that there are constants $c,\,c_1>0$ such that for all $t\in[0,T]$ it holds
\begin{align}\label{equ:coercive}
    \fe(t,\by,\psi)
    \geq 
    c\big(\alpha\|\by\|^p_{W^{1,p}(\Omega) } 
    + \|\det(\nabla \by)^{-1}\|_{L^q(\Omega)}^q
    + \gamma\|\by\|^{\beta}_{W^{2,\beta}(\Omega)}
    + \|\psi\|^2_{H^1(\Omega)} - c_1\big)
\end{align}
for all $(\by,\psi)\in\bfY \times\bfZ$.
\item 
For every $t\in[0,1]$,
the functional $\fe(t,\cdot,\cdot):\bfY \times\bfZ\to\R$
is weakly lower continuous
and its sublevel sets are weakly sequentially compact. 
\end{enumerate}
\end{lemma}

\begin{proof}
At first, we show coercivity.
With the notation in \eqref{def-Fp_Ogden}, with \eqref{equ:DBC_with_y} and with the definition of $\Wel$, $\Wpf$, $\Why$ in \eqref{def-Wel-0} and \eqref{def-W_double}, and for $g$ as in \eqref{equ:alles_g-OO}, we obtain 
\begin{align}
\label{equ:coerziv_esti}
\begin{split}
    &\fe (t, \by, \psi) 
    = \int_{\Omega}\Big( \Wel(\Fe) + \frac{a}{4}(\psi^2-1)^2  + \frac{b}{2}|\bF^{-T}\nabla \psi|^2
    + \Why(\nabla\bF)
    \Big)
    \dx  \\
    & \quad 
    \geq \int_{\Omega}\Big( 
    \frac{\alpha}{2p\overline{g}^p}|\nabla \bv|^p 
    + \frac{c\underline{g}^{dq}}{q}J^{-q} + \frac{a}{8}\psi^4 - \frac{a}{4} + \hat c|\nabla \psi|^2 \Big)\dx
    + \frac{c\gamma}{\beta}\|\nabla\bF\|_{L^{\beta}(\Omega)}^{\beta}\,, 
\end{split}
\end{align}
with $\bv(t,x)=\bvD(t,\by(x))$, $\bF=\nabla \bv$ and $J=\det(\bF)$,
where we used Young's inequality to estimate
$(\psi^2-1)^2\geq \frac{\psi^4}{2}-1$ as well as
\[
|\bF^{-T}\nabla \psi|^2
\geq {\varepsilon} |\nabla\psi|^2-\frac{\varepsilon^2}{4}|\bF|^2
\geq {\varepsilon} |\nabla\psi|^2-\frac{c_p\varepsilon^2}{4}(1+|\bF|^p)
\]
for some $c_p>0$,
and we choose $\varepsilon^2=\frac{2\alpha}{c_p p\overline{g}^p}$.
For the first term on the right-hand side of \eqref{equ:coerziv_esti} we
use~\eqref{equ:grad_v_NH}, \eqref{equ:BDC_3} and Poincar\'e's inequality
to estimate
\[
\begin{aligned}
\|\nabla \bv\|_{L^p(\Omega)}
\geq C\|\nabla \by\|_{L^p(\Omega)}
&\geq C(\|\nabla(\by-\id)\|_{L^p(\Omega)} - \|\nabla \id\|_{L^p(\Omega)})
\\
&\geq C(\|( \by - \id)\|_{W^{1,p}(\Omega)} - \|\nabla \id\|_{L^p(\Omega)})
\\
&\geq C\|\by\|_{W^{1,p}(\Omega)}-\tilde c\,.
\end{aligned}
\]
Moreover, $\nabla \bv(t)=\nabla_{\by}\bvD(t,\by)\nabla \by$ implies
\[
J=\det\big(\nabla_{\by}\bvD(t,\by)\big)\det(\nabla \by)
\leq C\det(\nabla \by)
\]
by~\eqref{equ:BDC_2},
and by the boundedness of $\Omega$ and Young's inequality we deduce
\[
\|\psi\|_{L^4(\Omega)}^4
\geq C\|\psi\|_{L^2(\Omega)}^4
\geq  C\big(\|\psi\|_{L^2(\Omega)}^2 - 1\big).
\]
Using the last three inequalities in~\eqref{equ:coerziv_esti} leads to
\begin{align}
\label{equ:coerziv_esti1}
\begin{split}    
\fe (t, \by, \psi) 
    & \geq c\big(\alpha\|\by\|_{W^{1,p}(\Omega)}^p  + \|\det(\nabla \by)^{-1}\|_{L^q(\Omega)}^q +  \|\psi\|_{W^{1,2}(\Omega)}^2 - c_1\big)+ \frac{c\gamma}{\beta}\|\nabla\bF\|_{L^{\beta}(\Omega)}^{\beta}\,. 
\end{split}
\end{align}
To treat the second-order term, we set 
$B\coloneqq\norm{\nabla^2\bvD}_{L^\infty((0,T)\times\Omega)}$ and use 
\eqref{equ:componentwise-zweite} 
to observe
\begin{align*}
    \|\nabla\bF\|_{L^{\beta}(\Omega)}
    & \geq \|\nabla_{\by}\bv_{\D}^T\nabla^2\by\|_{L^{\beta}(\Omega)}
    - \|\nabla\by^T\nabla^2_{\by\by}\bv_{\D}\nabla\by\|_{L^{\beta}(\Omega)}
    \\
    &\geq C\|\nabla^2\by\|_{L^{\beta}(\Omega)}
    - B\|\nabla\by\|_{L^{2\beta}(\Omega)}^{2}\,.
\end{align*}
Due to $p\geq 2\beta$, by Young's inequalities, it follows 
\[
    \|\nabla\by\|^{2\beta}_{L^{2\beta}(\Omega)} 
    \leq C(1+\|\nabla\by \|^p_{L^p(\Omega)})\,,
\]
and similarly, 
\[
    \|\by \|^p_{W^{1,p}(\Omega)} 
     \geq \frac{1}{2}\|\by \|^p_{W^{1,p}(\Omega)} 
     + c\|\by\|^{\beta}_{W^{1,\beta}(\Omega)}
    - \tilde{c}\,.
\]
Combining these inequalities with \eqref{equ:coerziv_esti1}, we obtain
\[
\begin{aligned}
\fe (t, \by, \psi) 
    &\geq (c\alpha-C\gamma B^\beta)\|\by\|_{W^{1,p}(\Omega)}  
    \\
    &\quad
    + c\big(\|\det(\nabla \by)^{-1}\|_{L^q(\Omega)}^q +  \|\psi\|_{W^{1,2}(\Omega)}^2
    + \|\by\|^{\beta}_{W^{1,\beta}(\Omega)}-c_1\big),
\end{aligned}
\]
which yields~\eqref{equ:coercive}
if  $\alpha^{-1}\gamma B^\beta$ is sufficiently small.

To show weak lower semicontinuity,
we consider $\fe (t,\by, \psi)$ as a functional 
in $(\bF,\psi)$.
Since $W$ is a convex function in $(\nabla\psi,\nabla\bF)$,
and weak convergence of a sequence $(\by_k)$ in $\bfY$ 
yields weak convergence of $(\bF_k)=(\nabla\bvD(t,\by_k)$ 
in $W^{1,\beta}(\Omega,\R^{d\times d}$ by Lemma~\ref{lemma:conv_gradient_strong},
we obtain weak lower semicontinuity of $\fe (t,\cdot,\cdot)$,
see~\cite[Thm. 8.16]{dacorogna2007direct} for instance.
The weak sequential compactness of sublevel sets now follows from the coercivity 
property by a standard argument.
\end{proof}

Next we examine the time regularity of the 
energy functional $\fe$.
For this, we use the decomposition
\[
\begin{aligned}
    \fe(t,\by,\psi) 
    &=\intO \Wel(\Fe) \dx 
    + \intO \Wpf(\psi,\nabla\psi,\bF) \dx
    + \intO \Why(\nabla\bF) \dx
    \\
    &\eqqcolon \feel(t,\by,\psi) + \fepf(t,\by,\psi) + \fehy(t,\by),
\end{aligned}
\]
where $\bF=\nabla_{\by}\bvD(t,\by)\nabla\by$, $\Fp=g(\psi)\I$ and $\Fe=\bF\Fp^{-1}$.

\begin{lemma}\label{OO_lemma:derivative_estimate}
Let the assumptions \ref{Ass0}, \ref{Ass1}, \ref{Ass2}, \ref{Ass6}, \ref{Ass7} be satisfied,
and
let $(\by,\psi)\in \bfY\times\bfZ$.
Then 
$\calF(\cdot,\by,\psi)\in C^1([0,T])$ with
\begin{align}\label{equ:time_deriv}
\begin{split}
    \partial_t\feel (t,\by,\psi) 
    & = \int_{\Omega}\partial_{\bA}\Wel(\Fe)\Fe^T:\left[\partial_t\nabla_{\by}\bvD(t,\by))(\nabla_{\by}\bvD(t,\by))^{-1}\right] \dx, \\
    \partial_t \fepf(t,\by,\psi)  
    & = \int_{\Omega}\partial_{\bF}\Wpf (\psi,\nabla\psi,\bF) :\partial_t\bF \dx,
    \\
    \partial_t \fehy(t,\by)  
    & = \int_{\Omega}\partial_{\bG}\Why(\nabla\bF )\tdots\partial_t\nabla\bF\dx\,.
\end{split}
\end{align}
Moreover, there exist
constants $c_0,\,c_1>0$ such that
\begin{equation}\label{00-equ:est_time_deriv} 
    |\partial_t\calF(t,\by,\psi)|
    \leq c_1\left(\calF(t,\by,\psi)+c_0\right)\,, 
\end{equation}
for every $t\in[0,T]$. 
Additionally, on sublevel sets,
$\fe(\cdot,\by,\psi)$ is Lipschitz continuous
and $\partial_t \fe(\cdot,\by,\psi)$ is uniformly continuous in the following sense:
Let $F>0$ and $\varepsilon>0$.
There exist $C_F>0$ and $\delta>0$ such that 
for all $(\by,\psi)\in\bfY\times\bfZ$ with $\calF(t,\by,\psi)\leq F$ 
for some $t\in[0,T]$, 
it holds 
\begin{equation}\label{equ:Lip_estimate}
    |\calF(t,\by,\psi) - \calF(s,\by,\psi)| \leq C_F|t-s|\,,
\end{equation}
for all $s\in[0,T]$,
and 
\begin{equation}\label{eq:unifcont.est}
    |\partial_t\calF(t,\by,\psi) - \partial_{t}\calF(s,\by,\psi)|
    < \varepsilon 
\end{equation}
for all $t,s \in[0,T]$ with $|t-s|<\delta$.
\end{lemma}
Estimate~\eqref{00-equ:est_time_deriv} 
shows that the power of the external forces
$\partial_t\calF$, in terms of the Dirichlet data, 
can be controlled by the energy.

\begin{proof}[Proof of \Cref{OO_lemma:derivative_estimate}]
We first note that for all $\by\in\bfY$
we have $\by\in \mathrm{C}^1(\overline{\Omega};\R^d)$,
and together with~\eqref{equ:prop-R_2} and \eqref{equ:alles_g-OO}
we conclude that
$\calF(t,\by,\psi)<\infty$ for all
$t\in[0,T]$ and $(\by,\psi)\in\bfY\times\bfZ$.
Moreover, $t\mapsto\calF(t,\by,\psi)$
is continuous due to regularity properties of $\bvD$ from~\ref{Ass6}.
To show the differentiability
and the estimate \eqref{00-equ:est_time_deriv}, 
we fix $t\in[0,T]$.
For $h\in\R$ such that $t+h\in[0,T]$,
the difference quotient is given by
\begin{align*}
    \frac{1}{h}\left(\feel(t+h,\by)-\feel(t,\by)\right)
    & = \int_\Omega \frac{1}{h}\int_0^h\partial_{\bA}\Wel(\Fe(t+\theta,x)):\partial_t(\Fe(t+\theta,x))\,\dtheta\dx\,,
\end{align*}
with $\Fe(t+h,x)=\frac{1}{g(\psi(x))}\nabla_{\by}\bvD(t+h,\by(x))\nabla\by(x)$. 
Using \eqref{equ:BDC_1}, we find
that the integrand on the right-hand side converges to
\[
\begin{aligned}
&\partial_{\bA}\Wel(\Fe(t,x)):\partial_t\bF_e(t,x)
\\
&\qquad
=\partial_{\bA}\Wel(\Fe(t,x))\Fe(t,x)^T:\partial_t\nabla_{\by}\bvD(t,\by(x)))(\nabla_{\by}\bvD(t,\by(x)))^{-1}
\end{aligned}
\]
as $h\to0$.
From \eqref{equ:stress_contral_KH2} and \ref{Ass7}, we further obtain an integrable upper bound such that
by dominated convergence we conclude the existence of $\partial_t\feel$
as in \eqref{equ:time_deriv} and estimate  \eqref{00-equ:est_time_deriv} for $\calF$ replaced with $\feel$.

For the phase-field energy, we can proceed in the very same way to determine
the time derivative $\partial_t\fepf$. 
From the formula~\eqref{eq:Wpfder.F}
we then conclude
\[
\begin{aligned}
|\partial_t\fepf(t,\by,\psi)|
&\leq b\int_\Omega \big| \partial_t\bF \bF^{-1}| \,|\bF^{-T}\nabla\psi|^2\,\dx
\\
&\leq \norm{\partial_t\nabla_{\by}\bvD(t,\by)(\nabla_{\by}\bvD(t,\by))^{-1}}_{L^\infty(\Omega)} \fepf(t,\by,\psi).
\end{aligned}
\]
Due to the bounds in~\ref{Ass7},
this yields estimate~\eqref{00-equ:est_time_deriv} for $\calF$ replaced with $\fepf$.  

We next calculate the time derivative of the hyperelasticity term $\fehy$ using the assumption~\ref{Ass2}.
For $t$ and $h$ as above, the difference quotient is given as 
\begin{align*}
    \frac{1}{h}\left(\fehy(t+h,\by)-\fehy(t,\by)\right)
    & = \int_{\Omega} \frac{1}{h}\int_0^h \!
    \partial_{\bG}\Why(\nabla\bF(t+\theta,x) )\tdots\partial_t(\nabla\bF(t+\theta,x))\dtheta\dx\,,
\end{align*}
with $\nabla\bF(t,x)=\nabla^2\bv(t,x)$ calculated in \eqref{equ:componentwise-zweite}. 
Clearly, the integrand converges to
\[
    \partial_{\bG}\Why(\nabla\bF(t,x) )\tdots\partial_t(\nabla\bF(t,x))
\]
as $h\to 0$,
due to the regularity of $\bvD$ from \eqref{equ:BDC_1}.
In order to perform the limit passage also in the spatial integral,
we need an integrable upper bound. 
For this purpose, we derive from~\eqref{equ:componentwise-zweite} that 
\[
\begin{split}
    \partial_t\left(\nabla\bF(t+\theta,\cdot)\right)_{ijk}
    & = {(\nabla\by_j)}^T{\partial_t(\nabla_{\by\by}^2\bv_{\D_i}(t+\theta,\by))}{(\nabla\by_k)}
    + \partial_t(\nabla_{\by}\bv_{\D_i}(t+\theta,\by))^T{(\nabla_{jk}^2\by)}\,.
\end{split}
\]
Due to the assumptions~\ref{Ass2} and~\ref{Ass7}, this implies
\[
\begin{aligned}
    \big|\partial_{\bG}\Why(\nabla\bF(t+\theta,x) )&\tdots\partial_t(\nabla\bF(t+\theta,x))\big|
    \leq  C(1+|\nabla\bF(t+\theta,x)|^{\beta-1}) (|\nabla\by|^2 + |\nabla^2\by|)
    \\
    &\leq C (|\nabla\by(x)|^2 + |\nabla^2\by(x)| + |\nabla\by(x)|^{2\beta} + |\nabla^2\by(x)|^{\beta})\,,
    \\
    &\leq C (1+ |\nabla\by(x)|^{2\beta} + |\nabla^2\by(x)|^{\beta})\,,
\end{aligned}
\]
which is an integrable bound
since $\by\in\bfY\subset W^{2,\beta}(\Omega;\R^d)$ with $\beta>d$
implies $\nabla^2\by\in L^{\beta}(\Omega;\R^{\dxd\times d})$ and $\nabla\by\in L^{2\beta}(\Omega;\R^{\dxd})$ 
by Sobolev embeddings. 
Hence, we can apply dominated convergence to conclude time differentiability of $\fehy$
with
$\partial_t\fehy$
given in \eqref{equ:time_deriv}. 
Using the previous estimate, we further find that
\[
|\partial_t\fehy(t,\by)|
    \leq \int_{\Omega}|\partial_{\bs{G}}\Why(\nabla\bF)\tdots\partial_t(\nabla\bF)|\dx
    \leq C\int_{\Omega} 1+ |\nabla\by(x)|^{2\beta} + |\nabla^2\by(x)|^{\beta}\dx.
\]
Since $2\beta\leq p$, Young's inequality implies 
\[
\|\nabla\by\|^{2\beta}_{L^{2\beta}(\Omega;\R^{\dxd})}
\leq \|\nabla\by\|^{2\beta}_{L^{p}(\Omega;\R^{\dxd})}
\leq C(1+ \|\nabla\by\|^{p}_{L^{p}(\Omega;\R^{\dxd})}),
\]
and with Lemma~\ref{lemma:estvy} we obtain
\[
\begin{aligned}
    |\partial_t\fehy(t,\by)|
    \leq C\big(1+ \|\nabla\by\|^{p}_{L^{p}(\Omega)}
    +\|\nabla\bF\|_{L^{\beta}(\Omega)}^{\beta}\big)
    \leq C(1+\feel(t,\by,\psi)+\fehy(t,\by))\,.
\end{aligned}
\]
In total, we conclude the full estimate \eqref{00-equ:est_time_deriv} for $t\in [0,T]$.

By Gronwall's inequality applied to the functions $t\mapsto\pm\left(\calF(t,\by,\psi)+c_0\right)$
we further conclude  
\begin{align}\label{equ:it_follows}
    |\calF(t,\by,\psi)- \calF(s,\by,\psi)|
    \leq \left(\exp{(c_1|t-s|)}-1\right)(\calF(t,\by,\psi)+c_0)
\end{align}
for all $s,t\in [0,T]$,
and the Lipschitz estimate \eqref{equ:Lip_estimate}
now follows from~\eqref{equ:it_follows} and the elementary estimate 
$e^{\tau}-1\leq \tau e^\tau$ for $\tau\in\R$
since $\calF(t,\by,\psi)\leq F$ implies
\[
   |\calF(t,\by,\psi)- \calF(s,\by,\psi)|
   \leq \left(\exp{(c_1|t-s|)}-1\right)(F+c_0)
   \leq c_1|t-s| \,e^{c_1 T}(F+c_0)\,.
\]

It remains to verify the asserted uniform continuity of $\partial_t\calF(\cdot,\by,\psi)$ on sublevel sets. 
We introduce 
\[
\begin{aligned}
    \bF_{{\mathrm e}}(t)&=\nabla \bvD(t,\by)\nabla \by  {\Fp^{-1}} 
    \,,\\
    \bL(t) &=(\nabla \partial_t \bvD(t,\by))[\nabla \bvD(t,\by)]^{-1}\,,\\
    \bK(\bA)&=\partial_{\bA}\Wel(\bA)\bA^T\,.
\end{aligned}
\]
Then $\partial_t\feel$ can be expressed as 
\[
    \partial_t\feel(t,\by,\psi) 
    = \int_{\Omega} \bK(\bF_{{\mathrm e}}(t))\Fp^{-1} :\bL(t)\dx\,.
\]
Let $\hat{\varepsilon}>0$ be arbitrary. 
From \eqref{equ:BDC_1} we deduce that $t\mapsto\nabla\bvD(t,\by)\in C(\overline{\Omega};\R^{\dxd})$ and $t\mapsto \bL(t)\in C(\overline{\Omega};\R^{\dxd})$ are uniformly continuous,
so that for any $\tilde\varepsilon>0$ 
there is $\tilde\delta>0$ such that 
$|t-s|<\tilde\delta$ implies
\begin{align*}
    |\bL(t,x)-\bL(s,x)| & \leq \tilde\varepsilon\,,\quad
    |\bF_{{\mathrm e}}(s,x)\bF_{\mathrm e}^{-1}(t,x)-\I|  \leq\tilde\varepsilon \,\,\text{for } x\in \Omega\,.
\end{align*}
Therefore, for any $\overline\varepsilon>0$,
we choose $\tilde\delta$ sufficiently small
such that
we can employ \eqref{equ:uniform_continuity_stresses} 
to ensure
\[
    |\bK(\bF_{{\mathrm e}}(t))-\bK(\bF_{\mathrm e}(s)) |
    =\big|\bK(\Fe(t)-\bK\big([\bF_{{\mathrm e}}(s)\bF_{{\mathrm e}}(t)^{-1}]\bF_{{\mathrm e}}(t)\big)\big|
    \leq\overline{\varepsilon}\big(1+\Wel(\bF_{{\mathrm e}}(t))\big)
\]
for $|t-s|<\tilde\delta$.
Using assumption~\ref{Ass6} and~\eqref{equ:stress_contral_KH2}, we then have 
\begin{equation*}
\begin{aligned}
    |\partial_t\feel(t,\by,\psi)&-\partial_t\feel(s,\by,\psi)| 
    \\
    & \leq C\int_{\Omega}\left( |\bK(\bF_{{\mathrm e}}(t))-\bK(\bF_{{\mathrm e}}(s)) |\,|\bL(s)| 
    + |\bK(\bF_{{\mathrm e}}(t))|\, |\bL(t)-\bL(s)|\right)\dx \\
    & \leq C\int_{\Omega}\overline{\varepsilon}\big(1+\Wel(\Fe(t))\big)
    + \tilde{\varepsilon}\big(c+\Wel(\Fe(t))\big)\dx \\
    &\leq C \left(\overline{\varepsilon}+\tilde{\varepsilon}\right)
    (1+\feel(t,\by,\psi))
    \leq \Hat{\varepsilon}\,
\end{aligned}
\end{equation*}
for $|t-s|<\tilde\delta$ 
if $\overline\varepsilon,\,\tilde\varepsilon>0$ are sufficiently small.
For the time derivative of $\fepf$ and $\fehy$, 
we proceed similarly, also using~\eqref{equ:prop-R_22} and~\eqref{stress_c_reg}.
In total we obtain~\eqref{eq:unifcont.est}, 
which finishes the proof.
\end{proof}

From previous results, we conclude the following convergence property 
of $\partial_t\calF$.

\begin{lemma} \label{lem:Fg.conv}
Let $\alpha^{-1}\gamma\norm{\nabla^2\bvD}_{L^\infty((0,T)\times\Omega)}^\beta<\delta_0$
with $\delta_0$ as in Lemma~\ref{lemma:Coercive_LowerSemi}.
For each $t\in[0,T]$ and all sequences $\psi_m\rightharpoonup\psi$ in $\bfZ$, $\by_m\rightharpoonup \by$ in $\bfY$
such that $\calF(t,\by_m,\psi_m)\to\calF(t,\by,\psi)$,
 we have
\[
\partial_t\calF(t,\by_m,\psi_m)\to \partial_t\calF(t,\by,\psi)
\quad\text{as }m\to\infty\,.
\]
\end{lemma}

\begin{proof}
By using \cite[Prop.~3.3]{francfort2006existence}, 
the statement is a direct consequence of the lower semicontinuity of
$\mathcal F(t,\cdot,\cdot)$ shown in Lemma~\ref{lemma:Coercive_LowerSemi}
and the uniform continuity of $\partial_t\calF$ on sublevel sets
as shown in Lemma~\ref{OO_lemma:derivative_estimate}.
\end{proof}

We next show the Gateaux differentiability of 
the free energy. 
Since we are interested in the 
derivative with respect to the displacement $\bv$ rather than 
the auxiliary variable $\by$,
we introduce the free energy as a function of $(\bv,\psi)$, namely 
\begin{equation}\label{eq:calEg}
\begin{aligned}
&\calE\colon W^{2,\beta}(\Omega;\R^d)\times H^1(\Omega)\to\R\cup\{\infty\},
\\
&\calE(\bv,\psi)
= \intO\Wel(\tfrac{1}{g(\psi)}\nabla\bv) +\Wpf(\psi,\nabla\psi.\bF)
+\Why(\nabla^2\bv) \dx\,. 
\end{aligned}
\end{equation}
Then it is clear that $\calE(\bv,\psi)=\fe(t,\by,\psi)$
if $\bv=\bvD(t,\by)$.
To obtain simple formulas,
we again use the notation $\partial_{\bF}\Wel(\Fe)$,
which denotes the derivative with respect to the full deformation gradient $\bF=\nabla\bv$, in contrast to $\partial_{\bA}\Wel(\Fe)$, which is the derivative of $\Wel$ with respect to the entire argument $\Fe=\bF\Fp^{-1}(\psi)$, see also Remark~\ref{rem:dAdFWel}.

\begin{lemma}\label{lem:Gateaux}
Let $\bv\in\bfU(t)$ and $t\in[0,T]$,
and let $\psi\in\bfZ$.
Then $\calE$ is Gateaux-differentiable in $(\bv,\psi)$, 
and for $(\bw,\zeta)\in W^{2,\beta}_{\DBC}(\Omega;\R^d)\times H^1(\Omega)$ 
it holds
\begin{align}\label{O-def-notation-derivatives}
\begin{split}
    \langle\D_{\bv}\calE(\bv,\psi), \bw \rangle 
    & = \intO \big(\partial_{\bF}\Wel(\Fe)+\partial_{\bF}\Wpf(\psi,\nabla\psi,\bF)\big):\nabla\bw +  \partial_{\bG}\Why(\nabla\bF)\tdots\nabla^2\bw\dx\,,
    \\
    \langle\D_{\psi}\calE(\bv,\psi), \zeta \rangle
    & = \intO \partial_{\psi}W(\bF,\nabla\bFg,\psi,\nabla\psi)\zeta + \partial_{\nabla\psi}\Wpf(\psi,\nabla\psi,\bF)\cdot\nabla\zeta\dx\,,
\end{split}
\end{align}
where more explicit formulas for 
$\partial_{\bF}\Wel$, $\partial_{\bF}\Wpf$, $\partial_{\psi}W$ and $\partial_{\nabla\psi}\Wpf$ 
are given in~\eqref{eq:Wder.F}, \eqref{eq:Wpfder.F}, \eqref{eq:Wder.psi} and \eqref{eq:Wder.dpsi}.
\end{lemma}

\begin{proof}
For $\bv\in\bfU(t)$, $\psi\in\bfZ$ and $\bw\in W^{2,\beta}_{\DBC}(\Omega;\R^d)$
we first consider the difference quotient
\[
\begin{aligned}
\frac{1}{h}&\big( \calE(\bv+h\bw,\psi)-\calE(\bv,\psi)\big)
\\
&=\int_\Omega \frac{1}{h}\int_0^h \partial_{\bA}\Wel\big(\tfrac{1}{g(\psi)}(\nabla\bv+\theta\nabla\bw)\big):\nabla\bw\tfrac{1}{g(\psi)} + \partial_{\bG}\Why(\nabla^2\bv+\theta\nabla^2\bw)\tdots\nabla^2\bw\dtheta\dx
\end{aligned}
\]
for $|h|>0$ small.
Due to Sobolov embeddings, 
we have $\bv,\bw\in C^{1}(\overline{\Omega})$,
so that $\nabla\bv$ and $\det(\nabla\bv)^{-1}$
as well as $\nabla\bv+\theta\nabla\bw$ and $\det(\nabla\bv+\theta\nabla\bw)^{-1}$
are uniformly bounded 
for $\theta$ sufficiently small.
Moreover, the bound~\eqref{equ:prop-R_22} of $\partial_{\bG} \Why$ yields
a uniform bound for 
$\partial_{\bG}R(\nabla^2\bv+\theta\nabla^2\bw)\tdots\nabla^2\bw$.
Therefore, we pass to the limit $h\to0$ by dominated convergence
and obtain the asserted derivative $\D_{\bv}\calE$.

Now let $\zeta\in H^1(\Omega)$ and $h>0$, 
and consider the difference quotient
\[
\begin{aligned}
&\frac{1}{h}\big(  \calE(\bv,\psi+h\zeta)-\calE(\bv,\psi)\big)
\\
&\ =\int_\Omega \frac{1}{h}\int_0^h \partial_{\psi}\Wel\big(\tfrac{1}{g(\psi+\theta\zeta)}\nabla\bv\big):\nabla\bv\tfrac{g'(\psi+\theta\zeta)}{g(\psi+\theta\zeta)^2}\zeta
+ \partial_\psi\Wpf(\psi+\theta\zeta,\nabla\psi+\theta\nabla\zeta,\nabla \bv)\zeta
\\
&\qquad\qquad\qquad\qquad\qquad\qquad\qquad\qquad\quad
+ \partial_{\nabla\psi}\Wpf(\psi+\theta\zeta,\nabla\psi+\theta\nabla\zeta,\nabla \bv)\cdot\nabla\zeta\dtheta\dx.
\end{aligned}
\]
For the first term, we can derive an integrable bound 
since $\nabla \bv$ and $\frac{1}{g}$ and $g'$ 
are bounded due to Sobolev embeddings and assumption~\ref{Ass6}.
An integrable bound for the second term can be derived from
the Sobolev embedding $H^1(\Omega)\hookrightarrow L^4(\Omega)$
and the definition of $\Wpf$ in \eqref{def-W_double}.
We can thus pass to the limit $h\to0$ by dominated convergence,
which yields the claimed formula for $\D_\psi\calE$.
\end{proof}

\section{The time-discretized system}
\label{sec:timediscrete}

Using an incremental minimization scheme, we derive solutions to a time discretization of the weak formulation of \eqref{equ:pde}. We show the existence of these minimizers by the direct method of calculus of variations.
Subsequently, we show regularity properties of the time interpolants associated to the discrete solutions,
and we derive uniform bounds.
For the whole section, we assume that the assumptions~\ref{Ass0}--\ref{Ass7}
are satisfied.

We consider an equidistant partition $0=t^0_M<t^1_M<\dots<t^M_M=T$ of the time interval $[0,T]$ with time step size $\tau_M=\tfrac{T}{M}$, where $M\in\N$ controls the discretization fineness, such that
\[
    t_M^m=m\tau_M=\frac{mT}{M}, \quad m\in\{0,1,\dots,M\}\,.
\]
We let $(\by^0,\psi^0)\in\bfY\times\bfZ$ be an initial datum.
To obtain the time-discrete approximate solutions,
for each $m\in\{1,\dots,M\}$ we iteratively solve the minimization problem
\begin{equation}\label{equ:incremental} 
    (\by_M^m,\psi_M^m)\in\underset{(\by,\psi)\in\bfY\times\bfZ_{\psi^0}}{\mathrm{argmin}}
    \mathbb{F}_M^m(\by,\psi)
\end{equation}    
for the nonlinear functional
\begin{equation}\label{funcitonal}
        \mathbb{F}_M^m(\by,\psi) = \fe(t_M^m,\by,\psi) + \tfrac{1}{2\tau_M}\|\psi-\psi_M^{m-1}\|^2_{\Tilde{V}_0} 
        + \tau_M \calV\left(t_M^m,\by^{m-1}_M,\tfrac{\by-\by^{m-1}_M}{\tau_M},\psi_M^{m-1}\right),
\end{equation}
where we set $\bfZ_{\psi^0}:=\{\phi\in H^1(\Omega)\,\vert\, \int_{\Omega}(\phi-\psi^0)\dx=0 \}$. 
The space $\Tilde{V}_0$ and the corresponding norm were defined in~\eqref{def-V0_H-1GF}.
We first show existence of these minimizers
under the smallness condition from Lemma~\ref{lemma:Coercive_LowerSemi}.

\begin{lemma}\label{lemma:time_disrete_sol}
Let $(\by^0,\psi^0)\in\bfY\times\bfZ$.
Then there is $\delta_0>0$ such that if
the smallness condition~\eqref{eq:smallness} holds, 
then for every $M\in\N$ and $m\in\{1,\dots,M\}$ there is a minimizer $(\by_M^m,\psi_M^m)\in {\bfY}\times\bfZ_{\psi^0}$ of
$\mathbb{F}_M^m$ defined in~\eqref{funcitonal}.
\end{lemma}

\begin{proof}[Proof of \Cref{lemma:time_disrete_sol}]
The set $\bfY\times\bfZ_{\psi_0}$ is a closed and convex subset of $W^{2,\beta}(\Omega;\R^d)\times H^1(\Omega)$.   
In \Cref{lemma:Coercive_LowerSemi} we showed that $\fe(t,\cdot,\cdot)$ has weakly sequentially compact sublevel sets 
if~\eqref{eq:smallness} holds for $\delta_0>0$ sufficiently small. 
The additional terms in \eqref{funcitonal} maintain these properties
due to their quadratic structure,
see also assumption~\ref{Ass5},
so that the direct method of calculus of variations yields 
the existence of minimizers of~\eqref{funcitonal}
for every $m\in\{1,\dots,M\}$ iteratively.
\end{proof}

\noindent
From the minimizers obtained in Lemma~\ref{lemma:time_disrete_sol}, we construct the piecewise constant interpolants $(\overline{\by}_M,\overline{\psi}_M)$, $(\underline{\by}_M,\underline{\psi}_M)$ 
and the piecewise linear interpolants $(\hat{\by}_M,\hat{\psi}_M)$
on $[0,T]$ by
\begin{subequations}\label{equ:all_interpolants}
    \begin{equation} 
    \begin{alignedat}{2}
    \overline{\psi}_M(t)           & = \psi_M^m && \text{for }t\in(t^M_{m-1},t_M^m]\,,\\
    \underline{\psi}_M (t)    & = \psi^{m-1}_M && \text{for }t\in[t^M_{m-1},t_M^m)\,,\\
    \overline{\by}_M(t)              & = \by_M^m  && \text{for }t\in(t^M_{m-1},t_M^m]\,,\\
    \underline{\by}_M (t)       & = \by^{m-1}_M  && \text{for }t\in[t^M_{m-1},t_M^m)\,, \label{def-Interpol-uy}\\
    \Hat{\psi}_M(t)     & = \frac{t-t_M^{m-1}}{\tau_M}\psi_M^m +\frac{t_M^m-t}{\tau_M}\psi^{m-1}_M \quad && \text{for }t\in(t^M_{m-1},t_M^m]\,,\\
    \Hat{\by}_M(t)     & = \frac{t-t_M^{m-1}}{\tau_M}\by_M^m +\frac{t_M^m-t}{\tau_M}\by^{m-1}_M \quad && \text{for }t\in(t^M_{m-1},t_M^m]\,,
    \end{alignedat}
    \end{equation} 
and the interpolant for the chemical potential is 
\begin{align}\label{equ:discret_mu}
    \mu_M(t)       & = -(-\Delta)^{-1}(\partial_t\Hat{\psi}_M(t)) + \lambda_M(t)\,,
\end{align}
where $\lambda_M(t)$ is a Lagrange multiplier 
originating from the conservation of mass
and determined by
\begin{align}\label{equ:lambda_M}
    \lambda_M(t) = \frac{1}{|\Omega|}\int_{\Omega}\partial_{\psi} \Wel(\nabla\overline{\bv}_M\Fp^{-1}(\overline{\psi}_M))+\partial_{\psi}\Wpf(\overline{\psi}_M, \nabla\overline{\psi}_M,\nabla\overline{\bv}_M)\dx\,.
\end{align}
\end{subequations}
Here we consider the inverse of the Laplace operator $-\Delta:V_0\to\tilde{V}_0$
introduced in~\eqref{equ:Inverse_laplace},
and since $\partial_t\Hat{\psi}_M(t)$ is mean-free by construction,
equation~\eqref{equ:discret_mu} yields a well-defined object 
$\mu_M(t)\in H^1(\Omega)$ for all $t\in[0,T]$.

We next derive the equations satisfied by the interpolants,
which correspond to Euler--Lagrange equations
for the solutions to the minimization problem~\eqref{equ:incremental}.
To simplify notation, we further introduce
\[
\begin{aligned}
    \overline{\bv}_M(t)&=\bvD(t,\overline{\by}_M(t)),  \qquad&
    \underline{\bv}_M(t)&=\bvD(t,\underline{\by}_M(t)), \qquad&
    \hat{\bv}_M(t)&=\bvD(t,\hat{\by}_M(t)),
    \\
    \overline{\bF}_M(t)&=\nabla\overline{\bv}_M(t), &
    \underline{\bF}_M(t)&=\nabla\underline{\bv}_M(t), &
    \hat{\bF}_M(t)&=\nabla\hat{\bv}_M(t),
\end{aligned}
\]
and $\overline{\bF}_{\mathrm e, M}(t)=\overline{\bF}_M(t)\bF_{\mathrm{p}}(\overline{\psi}(t))^{-1}$.

\begin{lemma}
\label{lem:weak_interp}
Let $M\in\N$
and let the interpolants \eqref{equ:all_interpolants}
be constructed from the solutions $(\by_M^m,\psi_M^m,\mu_M^m)_{m=1}^M$ to \eqref{equ:incremental}.
Then for all $t\in[0,T]$ and $\zeta\in H^1({\Omega})$, 
$\bw\in W^{2,\beta}_{\DBC}(\Omega;\R^d)$
the identities
    \begin{align}
    \int_{\Omega} \partial_t\Hat{\psi}_M(t)\zeta \dx   
         &= - \int_{\Omega}   \nabla \mu_M(t) \cdot \nabla\zeta \dx\,,\label{equ:weak_time} \\
    \begin{split}
        \int_{\Omega}  \mu_M(t) \zeta \dx   
        &= \int_{\Omega} \big[\partial_{\psi} W(\overline{\bF}_{M}(t),\nabla\overline{\bF}_{M}(t),\overline{\psi}_M(t),\nabla\overline{\psi}_M(t)) \zeta 
        \\
        &\qquad\qquad
        +\partial_{\nabla\psi} \Wpf(\overline{\psi}_M(t),\nabla\overline{\psi}_M(t),\overline{\bF}_M(t))\cdot\nabla\zeta \big]\dx\,, 
    \end{split}   
    \label{equ:weak_chem}
    \\
    \begin{split}
    0 = \int_{\Omega} \big[\partial_{\bF}\Wel&\left(\overline{\bF}_{\mathrm e, M}(t)\right):\nabla\bw  
    +\partial_{\bF}\Wpf\left(\overline{\psi}_M(t),\nabla\overline{\psi}_M(t),\nabla\overline{\bF}_M(t)\right):\nabla\bw
    \\
    +  \partial_{\bG}&\Why(\nabla\overline{\bF}_M(t))\tdots\nabla^2\bw
    + \partial_{\bs{\dot{F}}}V(\underline{\bF}_M(t), \partial_t{\hat{\bF}}_M(t), \underline{\psi}_M(t)) : \nabla\bw\big]\dx\,
    \end{split}
    \label{equ:weak_elast} 
\end{align}
hold,
and the energy-dissipation inequality 
\begin{align}\label{equ:weak_energy}
\begin{split}
    &\fe (t,\overline{\by}_M(t),\overline{\psi}_M(t)) 
    + \int_0^t\int_{\Omega} |\nabla\mu_M(s,x)|^2 \dx \ds 
    \\
    &\qquad\qquad
    + \int_0^t \intO V( \underline{\bF}_M(s,x),\partial_t{\hat{\bF}}_M(s,x),\underline{\psi}_M(s,x))\dx\ds
    \\
    &\qquad\qquad\qquad\qquad
    \leq \fe (0,{\by}^0,\psi^0) + \int_{0}^{t}\partial_t\fe (s,\underline{\by}_M(s),\underline{\psi}_M(s))\ds\, 
\end{split}
\end{align}
is valid.  
\end{lemma}

\begin{proof}[Proof of Lemma~\ref{lem:weak_interp}]
By~\eqref{equ:discret_mu}, we have
$(-\Delta)^{-1}(\partial_t\Hat{\psi}_M(t)) = \lambda_M(t) - \mu_M(t)$.
For $\zeta\in H^{1}(\Omega)$ we then obtain
\begin{align*}
    \intO \partial_t\Hat{\psi}_M(t,x)\zeta(x) \dx  
    & = \langle (-\Delta)^{-1}(\partial_t\Hat{\psi}_M(t)),\zeta \rangle_{\Tilde{V}_0}
    \\
    & =  \langle  (\lambda_M(t)-\mu_M(t)) , \zeta \rangle_{\Tilde{V}_0}\\
    & = -\intO \nabla \mu_M(t,x)\cdot\nabla\zeta(x) \dx\,,
\end{align*}
where we used $\lambda_M(t)\in\R$ in the last step. This shows \eqref{equ:weak_time}. 

By construction, $(\by^m_M, {\psi}^m_M)$
is a minimizer of $\bbF ^m_M$ in $\bfY\times\bfZ$
for $m=1,\dots,M$.
Due to the homeomorphism between $\bfY$ and $\bfU(t)$ induced by $\bvD(t)$, 
see Lemma~\ref{lemma:estvy},
the pair $q^m_M = (\bv^m_M, {\psi}^m_M)$ 
with $\bv^m_M\coloneqq\bvD(t^m_M,\by^m_M)$
is a minimizer of the functional
\[
        \mathbb{E}_M^m(\bv,\psi) = \calE(\bv,\psi) + \tfrac{1}{2\tau_M}\|\psi-\psi_M^{m-1}\|^2_{\Tilde{V}_0} 
        + \tau_M \intO V\big(\nabla\bv^{m-1}_M,\tfrac{\nabla\bv-\nabla\bv^{m-1}_M}{\tau_M},\psi_M^{m-1}\big)\dx
\]
in the class $\bfU(t^m_M)\times\bfZ$,
where $\calE$ was defined in~\eqref{eq:calEg}.
Since $\calE$ is Gateaux differentiable by Lemma~\ref{lem:Gateaux},
and the second term is quadratic in $\psi$,
we conclude 
$\langle\D_{\psi} \mathbb{E}^m_M(q^m_M),\zeta_1\rangle=0$ for all
$\zeta_1\in V_0$ and $m=1,\dots, M$,
so that
\begin{equation}\label{NEO:equ:ableitung_dings}
    0   
    = \langle \D_{\psi}\calE(\overline{\bv}_M(t),\overline{\psi}_M(t)),\zeta_1\rangle 
    + \langle \partial_t\Hat{\psi}_M(t),\zeta_1 \rangle_{\Tilde{V}_0}\,
\end{equation}
for $t\in[0,T]$.
We use $\intO\lambda_M(t)\zeta_1\dx=\lambda_M(t)\intO\zeta_1\dx = 0$ for $\zeta_1\in V_0$ 
and~\eqref{equ:discret_mu}, which allows us to conclude
\[
\begin{aligned}
    \langle \partial_t\Hat{\psi}_M(t),\zeta_1 \rangle_{\Tilde{V}_0}
    & = ((-\Delta)^{-1}\partial_t\Hat{\psi}_M(t), \zeta_1)_{L^2(\Omega)} \\
    & = (\lambda_M(t)-\mu_M(t), \zeta_1)_{L^2(\Omega)} \\
    & = - (\mu_M(t), \zeta_1)_{L^2(\Omega)}\,.
\end{aligned}
\]
Plugging this into \eqref{NEO:equ:ableitung_dings} gives   
\begin{align}\label{equ:NEO_first_derivative_zero}
    0 = 
    \langle \D_{\psi}\calE(\overline{\bv}_M(t),\overline{\psi}_M(t)),\zeta_1\rangle 
    - (\mu_M(t),\zeta_1)_{L^2(\Omega)}
\end{align}
for all $\zeta_1\in V_0$.
Moreover, for a constant test function $\zeta_2\equiv C\in\R$,
the definitions of $\mu_M$ and $\lambda_M$ in \eqref{equ:discret_mu} and \eqref{equ:lambda_M} imply 
\begin{align*}
 \langle\D _{\psi} &\calE(\overline{\bv}_M(t),\overline{\psi}_M(t)),\zeta_2\rangle 
      -(\mu_M(t),\zeta_2)_{L^2(\Omega)} \\
     & =  |\Omega|\lambda_M(t) C + ((-\Delta)^{-1}\partial_t\Hat{\psi}_M(t),\zeta_2)_{L^2(\Omega)} -  (\lambda_M(t),\zeta_2)_{L^2(\Omega)}
     = 0\,.
\end{align*}
Since every $\zeta\in H^1(\Omega)$ 
can be decomposed as $\zeta = \zeta_1+\zeta_2\in H^1(\Omega)$ 
with $\zeta_1\in V_0$ and $\zeta_2\equiv C\in\R$,
adding~\eqref{NEO:equ:ableitung_dings} and~\eqref{equ:NEO_first_derivative_zero} results in \eqref{equ:weak_chem}
due to~\eqref{O-def-notation-derivatives}.  

Combining the Gateaux differentiability of $\calE$
by Lemma~\ref{lem:Gateaux}
with the quadratic structure of $V$
from assumption~\ref{Ass5},
we conclude from the minimality of $q^m_M = (\bv^m_M, {\psi}^m_M)$ 
that $0=\D_{\bv}\mathbb E^m_M(q^m_M)$
for $m=1,\dots,M$,
so that
\[
    0 
    = \langle \D_{\bv}{\calE}(\overline{\bv}_M(t),\overline{\psi}_M(t)),\bw \rangle + \intO\partial_{\bs{\dot{F}}}V(\underline{\bF}_M(t),\partial_t{\hat{\bF}}_M(t),
    \underline{\psi}_M) : \nabla\bw \dx\,
\]
for $t\in[0,T]$
and all $\bw\in W^{2,\beta}_{\DBC}(\Omega;\R^d)$.
In virtue of~\eqref{O-def-notation-derivatives},
this shows \eqref{equ:weak_elast}.

The last step in this proof is to verify the energy-dissipation inequality. 
By minimality of $(\by^m_M,\,{\psi}_M^m)$, 
we have $\bbF^m_M(\by^m_M,\,{\psi}_M^m)\leq \bbF^m_M(\by^{m-1}_M,\,{\psi}_M^{m-1})$, 
which implies
\begin{align*}
    &\fe(t^m_M,\by^m_M,\psi_M^m) + \frac{1}{2\tau_M} \|{\psi^m_M-\psi_M^{m-1}}\|_{\Tilde{V}_0}^2  +\tau_M \calV\left(t_M^m,\by^{m-1}_M,\tfrac{\by^m_M-\by^{m-1}_M}{\tau_M},\psi_M^{m-1}\right)\\
    & \qquad
    \leq \fe(t^m_M,\by^{m-1}_M,\,{\psi}_M^{m-1})
    = \fe(t_M^{m-1},\by^{m-1}_M,\,{\psi}_M^{m-1}) + \int_{t_M^{m-1}}^{t^m_M}\partial_t{\fe}(s,\by_M^{m-1} ,\psi_M^{m-1})\ds\,.
\end{align*}
Summing up this estimate over $m=1,\,\dots,\,k$ results in
\begin{equation}\label{eq:edi.discrete}
    \begin{aligned}
    \fe(t_M^k,\by_M^k,\psi_M^k) + \sum_{m=1}^k & \bigg( \frac{\tau_M}{2} \bigg\Vert\frac{\psi^m_M-\psi_M^{m-1}}{\tau_M} \bigg\Vert_{\Tilde{V}_0}^2 + \tau_M \calV\left(t^m_M,\by_M^{m-1},\tfrac{\by^m_M-\by^{m-1}_M}{\tau_M^m},\psi_M^{m-1}\right) \bigg)\\ 
    & \leq {\fe}( \tO,\by^0,\psi^0) + \sum_{m=1}^k\int_{t_M^{m-1}}^{t^m_M}\partial_t{\fe}\left(s,\by_M^{m-1},\psi_M^{m-1}\right)\ds\,.
    \end{aligned}
\end{equation}
Using that $\|\frac{1}{\tau_M}(\psi^m_M-\psi_M^{m-1})\|_{\Tilde{V}_0}^2  = \|\nabla\mu^m_M\|^2_{L^2}$ by \eqref{equ:NEO_skp},
we thus have \eqref{equ:weak_energy} for $t\in[0,T]$. 
\end{proof}

From the energy-dissipation inequality we derive 
a useful bound of the interpolants.

\begin{lemma}\label{lemma:Gronwall}
In the setting of \Cref{lem:weak_interp} there exists a constant $C>0$,
independent of $M\in\N$, 
such that
\begin{equation}\label{equ:gron_est_integral}
    \begin{aligned}
    \fe (t,\overline{\by}_M(t),\overline{\psi}_M(t)) + \int_{0}^{t}\int_{\Omega}
    \frac{1}{2}|\nabla\mu_M(s)|^2  
    &+ V(\underline{\bF}_M(s),\partial_t{\hat{\bF}}_M(s),\underline{\psi}_M(s))
    \dx \ds 
    \\
    &\qquad\qquad\qquad \leq  C\left(\fe(0,\by^0,\psi^0)+1\right)\,
    \end{aligned}
\end{equation}
for all $t\in[0,T]$.
\end{lemma}

\begin{proof}[Proof of \Cref{lemma:Gronwall}] 

Starting with the discrete energy-dissipation inequality \eqref{eq:edi.discrete} and the estimate on $\partial_t\fe$ derived in \eqref{00-equ:est_time_deriv},
we obtain by an application of a discrete version of Gronwall's inequality to $\fe+c_0$ that
\begin{align*}
    \fe (t,\overline{\by}_M(t),\overline{\psi}_M(t)) 
    \leq  (\fe( \tO,\by^0,\psi^0)+c_0)\exp(c_1t)-c_0
    \leq  C(\fe( \tO,\by^0,\psi^0)+1)\,
\end{align*}
for a constant $C>0$ and all $t\in[0,T]$.
We can now use~\eqref{00-equ:est_time_deriv} and 
this estimate in~\eqref{equ:weak_energy}
to conclude from $0\leq\fe(t,q_M)$ that  
\begin{align*}
    \int_{0}^{t}\intO \frac{1}{2}|\nabla\mu_M(s)|^2
    &+ V(\underline{\bF}_M(s),\partial_t{\hat{\bF}}_M(s),\underline{\psi}_M(s)) \dx\ds  \,
    \\
    &\leq\fe( \tO,\by^0,\psi^0)
    + \int_{0}^{t}c_1\left(\fe(s,\underline{\by}_M(s),\underline{\psi}_M(s))+c_0\right)\ds  \\
    &\leq C(\fe( \tO,\by^0,\psi^0)+1)\,.
 \end{align*}
Combining the last two inequalities shows \eqref{equ:gron_est_integral}.
\end{proof}

\noindent
The next step is to show a priori bounds on the interpolants. These estimates are the crucial step to derive the convergence
to a weak solution to~\eqref{equ:pde} in Section~\ref{sec:finalproof}.

\begin{lemma}[A priori bounds]\label{lemma:energy_esti}
There is a constant $C>0$ such that for all $M\in\N$ the interpolants satisfy 
in \eqref{equ:all_interpolants} satisfy the
a priori estimates
\begin{subequations}
\begin{align}
    \|\overline{\by}_M\|_{L^{\infty}(0,T;W^{1,p}(\Omega; \R^d))}+  \,
    \|\underline{\by}_M\|_{L^{\infty}(0,T;W^{1,p}(\Omega; \R^d))} & \leq C\,,\label{equ:apri_y2}\\
    \|\overline{\by}_M\|_{ L^{\infty}(0,T;W^{2,\beta}(\Omega; \R^d))}+
    \|\underline{\by}_M\|_{L^{\infty}(0,T; W^{2,\beta}(\Omega; \R^d))} & \leq 
    C\,,\label{equ:apri_y2GG}\\
    \|\det(\nabla\overline{\by}_M)^{-1}\|_{L^\infty(0,T;L^\infty(\Omega))} +\|\det(\nabla\underline{\by}_M)^{-1}\|_{L^\infty(0,T;L^\infty(\Omega))} & \leq C\, 
    \label{equ:detbound}\\
    \|\partial_t\Hat{\by}_M\|_{L^2(0,T;H^1(\Omega;\R^d))}  & \leq C\,.\label{equ:visco_bound}
    \\
    \|\overline{\psi}_M\|_{L^{\infty}(0,T; H^1(\Omega))}+ \,
    \|\underline{\psi}_M\|_{L^{\infty}(0,T; H^1(\Omega))} & \leq C\,,\label{equ:apri_psi2}\\ 
    \|\partial_t\Hat{\psi}_M\|_{L^2(0,T;H^1(\Omega)^{*})} & \leq C\,,\label{equ:apri_psi2_hut}\\
    \|\mu _M\|_{L^2(0,T;H^1(\Omega))} & \leq C\,. \label{equ:apri_eta} 
\end{align}
\end{subequations}
\end{lemma}

To derive these bounds, we make use of two auxiliary results.
The first one is based on the following theorem by Healey and Krömer \cite{healey2009injective},
which ensures a positive lower bound for the determinant of the deformation gradient $\bF$,
and which will finally lead to the bound~\eqref{equ:detbound}.

\begin{lemma}\label{healy-kromer}
Let $\eta:{\mathrm{GL} }_+(d)\to\R_{\geq 0}$ be twice continuously differentiable and $H:\R^{\dxd\times d}\to\R_{\geq 0}$ be convex and continuously differentiable, such that there are $p,q,\beta\in(1,\infty)$ with
$\beta>d$, $p>2$, $q\geq\tfrac{\beta d}{(\beta-d)}$, 
and $c,K>0$ with 
\[
\begin{aligned}
        \eta(\bF)&\geq c|\bF|^p+c\det(\bF)^{-q} &&\text{for all }\bF\in{\mathrm{GL} }_+(d)\,,\\
        c|\bs{G}|^\beta&\leq H(\bs{G})\leq K(1+|\bs{G}|^\beta)&&\text{for all }\bG\in\R^{\dxd\times d}\,.
\end{aligned}
\]
Let $\widetilde{\calF}:\bfY\to\R$ be defined by
$\widetilde{\calF}(\by)=\int_{\Omega}\eta(\nabla\by) + H(\nabla^2\by)\dx$. 
Then for each $C_M>0$ there exists $C_{HK}>0$ such that for all $\by\in\bfY$
with $\widetilde{\calF}(\by)\leq C_M$
it holds
\begin{equation*} 
    \|\by\|_{W^{2,p}(\Omega;\R^d)}\leq C_{HK}, \qquad
    \det(\nabla\by(x))\geq C_{HK}^{-1}\,
\end{equation*} 
for all $x\in\Omega$. 
\end{lemma}
\begin{proof}
The original result was established in~\cite{healey2009injective}.
The version stated here can be found in~\cite[Theorem 3.1]{mielke2020thermoviscoelasticity}.
\end{proof}

The second result 
is a generalized Korn inequality,
which was essentially derived in~\cite{neff2002korn,pompe2003korn}.
It enables us to obtain the uniform bound \eqref{equ:visco_bound}
from the boundedness of the nonlinear viscosity term.

\begin{lemma}
\label{lemma:Korn}    
Let $\lambda\in(0,1)$
and $K>1$.
Then there exists $C_K>0$ such that for all 
$\bF\in C^\lambda(\Omega;\R^{d\times d})$ with 
$\norm{\bF}_{C^\lambda(\Omega)}\leq K$
and with $\det \bF(x)\geq K^{-1}$ for $x\in\Omega$,
it holds
\[
\forall \, \bw\in H^1(\Omega): \quad
\int_\Omega | \nabla\bw^T\bF+\bF^T\nabla\bw|^2\,\dx
+\int_\Gamma |\bw|^2\,\mathrm{d}\mathcal H^{d-1}
\geq C_K \norm{\bw}_{H^1(\Omega)}^2.
\]
\end{lemma}

\begin{proof}
The result for a constant $C_{K}$ depending on $\bF$
was shown in~\cite[Theorem 2.3]{pompe2003korn}.
To obtain a uniform constant that only depends on $K$,
we can repeat the argument from~\cite[Theorem 3.3]{mielke2020thermoviscoelasticity},
where the uniformity was shown for the special case $\bw|_\Gamma=0$.
\end{proof}

Observe that for Lemma~\ref{lemma:Korn}, besides the uniform 
lower bound on $\det \bF$, we also need
the uniform bound of $\bF$ in $C^\lambda(\Omega)$,
which will be ensured by the uniform bound
of $\by$, and thus of $\bv$, in $W^{2,\beta}(\Omega)$
and Sobolev embeddings.
The second-order derivatives from the hyperelasticity term 
are thus also necessary to ensure sufficient time regularity.

With the help of Lemma~\ref{healy-kromer} and Lemma~\ref{lemma:Korn},
we can now derive the a priori estimates from Lemma~\ref{lemma:energy_esti}.

\begin{proof}[Proof of \Cref{lemma:energy_esti}]
The coercivity estimate \eqref{equ:coercive} and the bound derived in \Cref{lemma:Gronwall} allow us to conclude the uniform bounds \eqref{equ:apri_y2}, \eqref{equ:apri_y2GG} and \eqref{equ:apri_psi2}. 
Further, from \Cref{lemma:Gronwall} we conclude the boundedness of ${(\nabla\mu_M)}_{M\in\N}$ in $L^2(0,T;L^2(\Omega))$,
which directly implies the bound~\eqref{equ:apri_psi2_hut} 
by~\eqref{equ:weak_time}.
Additionally, choosing $\zeta\equiv1$
in~\eqref{equ:weak_chem} gives the boundedness of $(\dashint_{\Omega}\mu_M(t)\dx)_{M\in\N}$
using~\eqref{equ:apri_psi2}.
Hence, we can use Poincare`s inequality to conclude \eqref{equ:apri_eta}.

To derive the lower bound on the Jacobians 
stated in~\eqref{equ:detbound},
we use Lemma~\ref{healy-kromer}
with
$
\eta(\bF):=|\bF|^p+\det(\bF)^{-q}$ and $H(\bG):=|\bG|^\beta$.
Then the coercivity estimate~\eqref{equ:coercive}
and Lemma~\ref{lemma:Gronwall}
imply that $\widetilde{\calF}(\overline\by_M(t))$
and $\widetilde{\calF}(\underline\by_M(t))$
are bounded uniformly in $t\in[0,T]$ and $M\in\N$.
Now Lemma~\ref{healy-kromer} yields~\eqref{equ:detbound}.

To prove \eqref{equ:visco_bound},
we use the generalized Korn inequality from Lemma~\ref{lemma:Korn}.
By~\eqref{equ:apri_y2GG}, Lemma~\ref{lemma:estvy} and Sobolev embeddings,
we have a uniform upper bound for
$\norm{\underline{\bF}_M(s)}_{C^{\lambda}(\overline{\Omega})}$
for $\lambda=1-d/p$,
and~\eqref{equ:detbound} 
yields a uniform lower bound for $\det(\underline{\bF}_M(s))$.
After estimating $V$ from below by~\eqref{equ:fuer_korn},
we can thus use Lemma~\ref{lemma:Korn} to obtain 
\begin{align*} 
    \int_0^T\intO& V(\underline{\bF}_M(s),\partial_t\hat{\bF}_M(s),\underline{\psi}_M(s))\dx\ds \\
    & \geq  c\int_0^T\intO|\partial_t\hat{\bF}_M(s)^T\underline{\bF}_M(s)+\underline{\bF}_M(s)^T\partial_t\hat{\bF}_M(s)|^2\dx\ds\\
    & \geq  c\int_0^T\Big(\|\partial_t\hat{\bv(s)}_M\|_{H^1(\Omega;\R^d)}^2\ds
    -\int_\Gamma |\partial_t\bvD(s)|^2\,\mathrm{d}\mathcal H^{d-1}\Big)\ds
    \\
    &\geq c(\|\partial_t\hat{\by}_M\|_{L^2([0,T];H^1(\Omega;\R^d))}^2-1)\,,
\end{align*}
where the last estimate follows from Lemma~\ref{lemma:estvy}
and \eqref{equ:BDC_2}.
Due to 
Lemma~\ref{lemma:Gronwall}, the left-hand side is bounded,
and
we obtain~\eqref{equ:visco_bound}.
This concludes the proof. 
\end{proof}

\section{Passage to the time-continuous limit}
\label{sec:finalproof}

To show existence of weak solutions to~\eqref{equ:pde},
we pass to the limit along the time-discrete solutions
constructed in Section~\ref{sec:timediscrete}.
At first, we show existence of a subsequence
with suitable convergence properties.

\begin{lemma}[Convergence of the interpolants]\label{lemma:interpol_conv}
Let~\ref{Ass0}--\ref{Ass7} be satisfied,
and let 
the assumptions of \Cref{lemma:energy_esti} be satisfied. Then we find (not relabeled) subsequences of the introduced interpolants, a pair $(\byg,\psig):[0,T]\to\bfY\times\bfZ$ and 
a function $\mug\in L^2(0,T;H^1(\Omega))$ such that 
\begin{subequations}\label{equ:conv_weak}
\begin{align}
    \hat{\by}_M,\,\overline{\by}_M,\, \underline{\by}_M &
    \xrightharpoonup{*}
     \byg
    && \text{in }L^{\infty}(0,T;W^{2,\beta}(\Omega;\R^d)) \,, \label{equ:y_weak} \\
    \hat{\by}_M &\rightharpoonup \byg
    &&\text{in } H^1(0,T;H^1(\Omega;\R^d))\,, \label{equ:y_derivative}
    \\    
    \hat{\by}_M
    & \to \byg
    && \text{in }C([0,T],C^{1,\lambda}(\overline{\Omega};\R^{d}))\, \text{ for all }\lambda\in(0,1-\tfrac{d}{\beta})\,,\label{equ:y_hut_stark} \\
    \hat{\by}_M,\,\overline{\by}_M,\underline{\by}_M
    & \to \byg
    && \text{in }L^\infty([0,T],W^{1,r}(\Omega;\R^{d}))
    \text{ for all } r\in[1,\infty),
    \label{equ:y_stark}
    \\
    \overline{\by}_M(t),\, \underline{\by}_M(t) &
    \rightharpoonup
     \byg(t)
    && \text{in }W^{2,\beta}(\Omega;\R^d) \text{ for all } t\in[0,T]\,, \label{equ:y_weak_ptw} \\
    \intertext{and}
    \Hat{\psi}_M,\,\overline{\psi}_M ,\,\underline{\psi}_M 
    & \xrightharpoonup{*}\psig 
    && \text{in }L^{\infty}(0,T;H^1(\Omega))\,, \label{equ:conv_psi_linf}\\
    \Hat{\psi}_M
    &  \rightharpoonup \psig
    && \text{in }H^1(0,T;H^1(\Omega)^*)\,,\label{equ:conv_hat_psi}\\
    \Hat{\psi}_M,\,\overline{\psi}_M ,\,\underline{\psi}_M 
    & \to\psig 
    && \text{in }L^{\infty}(0,T;H^1(\Omega)^*)\,, \label{equ:conv_psi_strong}\\
    \Hat{\psi}_M(t),\,\overline{\psi}_M(t),\underline{\psi}_M(t)
    & \rightharpoonup\psig
    &&\text{in } H^1(\Omega)\text{ for all }t\in[0,T]\,,\label{equ:cong_psi_l2}\\ 
    \mu_M
    & \rightharpoonup \mug
    && \text{in }L^2(0,T;H^1(\Omega))\,\label{equ:mu} 
\end{align}
\end{subequations}
as $M\to\infty$. 
Moreover, we have
\begin{subequations}
\label{eq:conv-special}
    \begin{align}    
    \nabla^2\overline{\bv}_M
    &\to\nabla^2\bvg \quad
    && \text{in }L^{\beta}(0,T;L^{\beta}(\Omega;\R^{\dxd\times d}))\,, \label{o-special-1}\\
    \overline{\psi}_M,\underline{\psi}_M
    &\to\psig\quad
    && \text{in }L^r(0,T;H^1(\Omega)) \text{ for all } r\in[1,\infty)\,,\label{o-special-2}
    \end{align}
\end{subequations}
where $\overline{\bv}_{M}(t,x):=\bvD(t,\overline{\by}_M(x))$ and $\bvg(t,x):=\bvD(t,\byg(x))$ for $(t,x)\in[0,T]\times\Omega$,
and it holds
\begin{equation}\label{eq:conv-energy}
    \fe(t,\by_M,\psi_M)\to\fe(t,\by,\psi) \qquad \text{ for all } t\in[0,T].
\end{equation}
\end{lemma}

\begin{proof}
While the prove of the convergence properties in~\eqref{equ:conv_weak}
completely relies on the a priori bounds established in Lemma~\ref{lemma:energy_esti},
for~\eqref{eq:conv-special} we exploit the specific structure of the underlying equations.

We start with the convergence properties of $(\hat{\by}_M)_M$.
By \eqref{equ:apri_y2GG} and \eqref{equ:visco_bound} we have
\begin{align*}
    \|\hat{\by}_M\|_{L^{\infty}(0,T;W^{2,\beta}(\Omega))} 
    & \leq C (\|\overline{\by}_M\|_{L^{\infty}(0,T;W^{2,\beta}(\Omega))}  +\|\underline{\by}_M\|_{L^{\infty}(0,T;W^{2,\beta}(\Omega))}) \leq C\,,\\
    \|\partial_t\hat{\by}_M\|_{L^2(0,T;H^1(\Omega))} 
    & \leq C\,.
\end{align*}
Hence, there exists $\by\in L^{\infty}(0,T;W^{2,\beta}(\Omega;\R^d))\cap H^1(0,T;H^1(\Omega;\R^d))$
such that~\eqref{equ:y_derivative} and
\begin{equation}
\label{eq:limityhat_proof}
\hat{\by}_M \xrightharpoonup{\ast} \byg\quad \text{in } L^{\infty}(0,T;W^{2,\beta}(\Omega;\R^d))
\end{equation}
hold
as $M\to\infty$  (for a not relabeled subsequence).
Using the Aubin-Lions lemma (e.g., see~\cite[Lemma 7.7]{roubivcek2013nonlinear}) and the compact embedding $W^{2,\beta}(\Omega)\hookrightarrow C^{1,\lambda}(\overline{\Omega})$
for $0<\lambda<1-\tfrac{d}{\beta}$,
we find~\eqref{equ:y_hut_stark}.
Furthermore, \eqref{equ:apri_y2GG} yields
elements $\overline\byg, \underline{\by}\in L^{\infty}(0,T;W^{2,\beta}(\Omega; \R^d))$ and
(not relabeled) subsequences such that
\begin{align}\label{equ:proof_limits_gleich}
    \overline{\by}_M\xrightharpoonup{*} \overline\byg,
    \qquad
    \underline{\by}_M\xrightharpoonup{*} \underline\byg\quad \text{in }L^{\infty}(0,T;W^{2,\beta}(\Omega; \R^d))\,.
\end{align}
To show that all three limit functions $\overline\byg$, $\underline\byg$ and $\byg$ coincide, 
let $t\in[0,T]$. 
Then there is $m\in\{1,\dots,M\}$ such that $t\in(t_M^{m-1},t^m_M]$. 
Then $\overline{\by}_M(t)=\overline{\by}_M(t^m_M)=\hat{\by}_M(t^m_M)$, and 
from~\eqref{equ:visco_bound} we obtain
\begin{equation}
\begin{aligned}
    \|\overline{\by}_M(t) - \hat{\by}_M(t)\|_{H^1(\Omega)} 
    & = \|\overline{\by}_M(t^m_M) - \hat{\by}_M(t)\|_{H^1(\Omega)}
    \leq \int_t^{t^m_M}\|\partial_t\hat{\by}_M(s)\|_{H^1(\Omega)}\,\mathrm{d}s \\
    & \leq C(t_M^m-t)^{\tfrac{1}{2}}\|\partial_t\hat{\by}_M\|_{L^2(0,T;H^1(\Omega))}
    \leq \tau_M^{\tfrac{1}{2}} C
\end{aligned}
\label{est:y.limit.same}
\end{equation}
for a.a.~$t\in(0,T)$,
so that $\|\overline{\by}_M - \hat{\by}_M\|_{L^\infty(0,T;H^1(\Omega))}\to 0$ as $M\to\infty$. 
Analogously, we find $\|\underline{\by}_M - \hat{\by}_M\|_{L^\infty(0,T;H^1(\Omega))}\to 0$.
Since we have already shown $\eqref{equ:y_hut_stark}$,
which implies the strong convergence $\hat\by_M\to\by$ in $L^2(0,T;H^1(\Omega))$,
the sequences $(\overline{\by}_M)_M$ and $(\underline{\by}_M)_M$ 
also converge strongly in this space towards the same limit $\by$.
Together with~\eqref{eq:limityhat_proof} and~\eqref{equ:proof_limits_gleich},
we thus conclude~\eqref{equ:y_weak}.
Moreover, for $r\in(2,\infty)$ we further deduce by interpolation that
\[
\begin{aligned}
\norm{\overline{\by}_M-\byg}_{L^\infty(0,T;W^{1,r}(\Omega))}
&\leq C \norm{\overline{\by}_M-\by}_{L^\infty(0,T;H^1(\Omega))}^{2/r}
\norm{\overline{\by}_M-\by}_{L^\infty(0,T;W^{1,\infty}(\Omega))}^{1-2/r}
\\
&\leq C\norm{\overline{\by}_M-\by}_{L^\infty(0,T;H^1(\Omega))}^{2/r} 
\norm{\overline{\by}_M-\by}_{L^\infty(0,T;W^{2,\beta}(\Omega))}^{1-2/r}
\to 0
\end{aligned}
\]
as $M\to\infty$
since the second factor is bounded due to~\eqref{equ:y_weak}.
As the same argument can be repeated for $(\underline{\by}_M)$ and $(\hat{\by}_M)$,
we obtain~\eqref{equ:y_stark}.
Additionally, \eqref{equ:apri_y2GG} yields the 
pointwise boundedness of $(\overline{\by}_M(t))_{M\in\N}$ in $W^{2,\beta}(\Omega)$ for all $t\in[0,T]$,
so that we obtain a ($t$-dependent) subsequence $(\overline{\by}_{M_k}(t))_{k\in\N}$
with $\overline{\by}_{M_k}(t)\rightharpoonup x_t$
in $W^{2,\beta}(\Omega)$ for some $x_t\in W^{2,\beta}(\Omega)$.
In particular, this weak convergence transfers to $H^1(\Omega)$,
and~\eqref{est:y.limit.same} implies $\hat\by_{M_k}(t)\rightharpoonup x_t$ in $H^1(\Omega)$.
From~\eqref{equ:y_hut_stark} we conclude $x_t=\by(t)$.
In particular, this limit does not depend on the chosen subsequence,
and we conclude $\overline{\by}_{M}(t)\rightharpoonup \by(t)$
as $M\to\infty$.
Arguing similarly for $(\underline{\by}_{M}(t))_{k\in\N}$,
we arrive at~\eqref{equ:y_weak_ptw}.

Next we address the interpolants for the phase field.
Similarly to before, 
we deduce the boundedness of $(\hat{\psi}_M)_{M\in\N}$ in $L^{\infty}(0,T;H^1(\Omega))$
from the corresponding boundedness 
of $(\overline{\psi}_M)_{M\in\N}$, $(\underline{\psi}_M)_{M\in\N}$
by~\eqref{equ:apri_psi2}.
By \eqref{equ:apri_psi2_hut} we know that $(\partial_t\hat{\psi}_M)_{M\in\N}$ is 
also bounded in $L^2(0,T;H^1(\Omega)^*)$,
so that there exists a (not relabeled) subsequence and $\psi\in L^\infty(0,T;H^1(\Omega))\cap L^2(0,T;H^1(\Omega)^*)$ 
with
$\Hat{\psi}_M\xrightharpoonup{*}\psig$ in $L^{\infty}(0,T;H^1(\Omega))$
and \eqref{equ:conv_hat_psi}.
From the Aubin-Lions lemma 
we now conclude
\begin{align}\label{l2l2_hut}
    \Hat{\psi}_M \to \psig\quad\text{in }C(0,T;L^2(\Omega))\quad\text{as }M\to\infty\,.
\end{align}
From \eqref{equ:apri_psi2} we further obtain elements $\overline\psig,\underline\psig
\in L^{\infty}(0,T;H^1(\Omega))$ 
and a subsequence 
such that there exist 
$\overline\psig,\underline\psig
\in L^{\infty}(0,T;H^1(\Omega))$ 
with
\begin{equation}\label{same-limits}
\overline{\psi}_M\xrightharpoonup{*}\overline\psig\,,\,\,\underline{\psi}_M\xrightharpoonup{*}\underline\psig\quad\text{in }L^{\infty}(0,T;H^1(\Omega))\,.
\end{equation}
To show that the limits $\psi$, $\overline\psi$ and $\underline\psi$ coincide,
we argue as above to first conclude
\[
    \|\overline{\psi}_M(t) - \Hat{\psi}_M(t)\|_{H^1(\Omega)^*}
    \leq C(t^m_M-t)^{\frac{1}{2}} \|\partial_t\Hat{\psi}_M  \|_{L^2(0,T;H^1(\Omega)^*)}
   \leq \tau_M^{\frac{1}{2}}C \,
\]
for $t\in(t^{m-1}_M,t^m_M]$, where we used~\eqref{equ:apri_psi2_hut}.
This implies 
$\|\overline{\psi}_M - \Hat{\psi}_M\|_{L^\infty(0,T;H^1(\Omega)^*)}
\to 0$, 
and we obtain $\|\underline{\psi}_M - \hat{\psi}_M\|_{L^\infty(0,T;H^1(\Omega)^*)}\to 0$ 
 as $M\to\infty$ in the same way.
Now,~\eqref{l2l2_hut} yields~\eqref{equ:conv_psi_strong},
so that \eqref{same-limits} implies $\psi=\overline{\psi}=\underline{\psi}$
and thus~\eqref{equ:conv_psi_linf}.
To conclude the pointwise convergence stated in \eqref{equ:cong_psi_l2},
we observe that
\[
\sum_{m=1}^M\|\psi_m^M-\psi_{m-1}^M\|_{H^1(\Omega)^\ast}
\leq C\sum_{m=1}^M \tau_M\|\frac{\psi_m^M-\psi_{m-1}^M}{\tau_M}\|_{\Tilde{V}_0}
= C\int_0^T\|\nabla\mu_M(s)\|_{L^2(\Omega)}\ds
\leq C\,,
\]
by estimate~\eqref{equ:apri_eta}.
This shows that $\hat{\psi}_M$, $\overline{\psi}_M$ and $\underline{\psi}_M$ have uniformly bounded total variation
as $H^1(\Omega)^\ast$-valued functions.
Hence, by a vector-valued version of Helly's selection principle,
see~\cite[Thm. B.5.10]{mielke2015rate} for instance, we 
obtain~\eqref{equ:cong_psi_l2}.
The convergence of $(\mu_M)_M$ stated in~\eqref{equ:mu} follows directly from estimate~\eqref{equ:apri_eta}
after selection of a suitable subsequence.

Next we show~\eqref{o-special-1}.
We set 
$\overline{\bF}_{{\mathrm e},M}=\overline{\bF}_M\Fp^{-1}(\overline{\psi}_M)$ for $\overline{\bF}_M=\nabla_{\by}\bvD(t,\overline{\by}_M)\nabla\overline{\by}_M$.
By the convexity of $\Why$, we have
\begin{equation}\label{estimate_beta}
    \int_{\Omega_T} \big(\Why(\nabla\overline{\bF}_{M})-\Why(\nabla\bFg)\big) \dxdt
    \leq\int_{\Omega_T}\partial_{\bG} \Why(\nabla\overline{\bF}_{M})\tdots(\nabla\overline{\bF}_{M}-\nabla\bFg)\dxdt\,.
\end{equation}
Due to~\eqref{equ:y_hut_stark},
we have the uniform convergence $\hat\by_M\to\byg$ and $\nabla\hat\by_M\to\nabla\by$ in $[0,T]\times\overline\Omega$,
which yields the uniform convergence $\overline\by_k\to\byg$ and $\nabla\overline\by_k\to\nabla\byg$ in $[0,T]\times\overline\Omega$.
By continuity of $\nabla \bvD$, this shows that $\overline{\bF}_M\to \bFg$ uniformly in $[0,T]\times\overline\Omega$ as $M\to\infty$.
Using $\bw=\overline{\bv}_{M_k}(t)-\bvg(t)\in W^{2,\beta}_{\DBC}(\Omega;\R^d)$ 
as a test function in the time-discrete Euler--Lagrange equation~\eqref{equ:weak_elast}
and integrating in time, we can replace the right-hand side of~\eqref{estimate_beta} 
to conclude
\begin{align*}
&\limsup_{M\to\infty} \int_{\Omega_T} \big(\Why(\nabla\overline{\bF}_{M})-\Why(\nabla\bFg)\big) \dxdt
\\
&\qquad
\leq - \lim_{M\to\infty}\int_{\Omega_T}
\big(\partial_{\bF}\Wel(\overline{\bF}_{{\mathrm e},{M}})+ \partial_{\bs{\dot{F}}}V(\overline{\bF}_M,\partial_t{\hat{\bF}}_M)\big):\big(\overline{\bF}_{M}-\bFg\big) \dxds  =0 \,.
\end{align*}
Here we use 
growth conditions for $\Wel$ and $\partial_{\bs{\dot{F}}}V$
due to~\eqref{equ:prop-R_1} and~\ref{Ass5}
and the boundedness of $\partial_t\hat{\bF}_M$
in $L^2(0,T;L^2(\Omega;\R^{d\times d}))$.
From the weak convergence $\nabla\overline{\bF}_{M_k}\rightharpoonup\nabla\bFg$ in $L^{\beta}((0,T)\times\Omega;\R^{\dxd\times d})$,
which is due to~\eqref{equ:y_weak} and Lemma~\ref{lemma:estvy},
and the convexity of $\Why$,
we further conclude 
\[
\liminf_{M\to\infty}\int_{\Omega_T} \big(\Why(\nabla\overline{\bF}_{M})-\Why(\nabla\bFg)\big) \dxdt \geq 0.
\]
In summary, this yields the convergence 
$\lim_{M\to\infty}\int_{\Omega_T} \Why(\nabla\overline{\bF}_{M})=\int_{\Omega_T} \Why(\nabla\bF)$.
Together with the weak convergence
$\nabla\overline{\bF}_{M_k}\rightharpoonup\nabla\bFg$ in $L^{\beta}((0,T)\times\Omega;\R^{\dxd\times d})$,
the strict convexity of $\Why$ and the growth condition~\eqref{equ:prop-R_2},
this implies the strong convergence~\eqref{o-special-1},
see \cite[Theorem X.2.4]{visintin1996phasetransitions} for instance.

To show \eqref{o-special-2}, we test \eqref{equ:weak_chem} with $\zeta=\overline{\psi}_M(t)$, integrate over time and pass to the limit such that we obtain 
\begin{align*}
    \int_{\Omega_T}|\overline{\bF}_M^{-T}\nabla\overline{\psi}_M|^2\dxdt         
    & \to -\int_{\Omega_T} \big(\partial_{\psi}\partial_{\psi}W(\bF,\nabla\bF,\psig,\nabla\psig)\,\psig+\mug \psig\big)\dxdt\,,
\end{align*}
as $M\to\infty$. 
Here we use that $\partial_\psi\Wpf(\overline{\psi}_M,\nabla\overline{\psi}_M,\overline{\bF}_M)$ is uniformly bounded in $L^2$ due to Sobolev embeddings. Further, due to $\partial_{\psi}\Wel(\overline{\bF}_{\mathrm e,M})=\partial_{\bF}\Wel(\overline{\bF}_{\mathrm e,M})\bF^T:\partial_{\psi}(\Fp^{-1}(\psi))$, the growth condition \eqref{equ:stress_contral_KH2} and the assumptions on $g$,
we have a uniform bound and can pass to the limit $M\to\infty$.
Similarly, testing \eqref{equ:weak_chem} with $\zeta=\psig(t)$, integrating over time
and passing to the limit by using the uniform convergence 
$\overline{\bF}_M\to\overline{\bF}$ yields
\begin{align*} 
    \int_{\Omega_T}\!|\overline{\bF}^{-T}\nabla\psig|^2\dxdt         
    & = -\int_{\Omega_T} \big(\partial_{\psi}W(\bF,\nabla\bF,\psig,\nabla\psig)\,\psig+\mug\psig\big)\dxdt
    \\
    &=\lim_{M\to\infty}\int_{\Omega_T}\!|\nabla\overline{\psi}_M|^2\dxdt   
\end{align*}
by the previous identity.
Together with the weak$^*$ convergence from~\eqref{equ:conv_psi_linf}, this gives the strong convergence 
$\overline{\bF}_M^{-T}\overline{\psi}_M\to\overline{\bF}^{-T}\psig$ 
and thus $\overline{\psi}_M\to\psig$ in $L^2(0,T;H^1(\Omega))$ as $M\to\infty$.
In particular, this also implies the strong convergence $\overline{\psi}_M\to\psig$
in $L^r(0,T;H^1(\Omega))$ for $r\in[1,2)$.
For $r\in(2,\infty)$, interpolation yields
\[
\|\overline{\psi}_M-\psig\|_{L^r(0,T;H^1(\Omega))}
\leq \|\overline{\psi}_M-\psig\|_{L^2(0,T;H^1(\Omega))}^{2/r}
\|\overline{\psi}_M-\psig\|_{L^\infty(0,T;H^1(\Omega))}^{(r-2)/r}
\to 0
\]
as $M\to\infty$
since the second factor is bounded due to~\eqref{equ:conv_psi_linf}.
In a similar way, we further conclude $\underline{\psi}_M\to\psig$ in $L^r([0,1];H^1(\Omega))$ as $M\to\infty$.
In total, this yields \eqref{o-special-2}.

Finally, \eqref{eq:conv-energy} directly follows from the convergence properties in
\eqref{equ:conv_weak} and \eqref{eq:conv-special}.
\end{proof}
 
Based on the solutions of the time-discrete Euler--Lagrange equations, we now perform the passage to vanishing time step size $\tau_M\to0$, or equivalently $M\to\infty$,
to obtain a weak solution to~\eqref{equ:pde}.
 
\begin{proof}[Proof of \Cref{Thm:main_result_fast}]  
The strategy of the proof is to pass to the limit $M\to\infty$ after integrating the weak formulations from \Cref{lem:weak_interp} over the time interval $[0,T]$,
which is based on the the convergence result in Lemma~\ref{lemma:interpol_conv}. 
We will discuss each limit transition individually.

Since the interpolants occur in \eqref{equ:weak_time}
in a linear way,
after integration over the time interval $[0,T]$, 
we can directly pass to the limit $M\to \infty$,
using the convergence properties \eqref{equ:conv_hat_psi} and \eqref{equ:mu},
to obtain~\eqref{equ:time}.

Similarly, the left-hand side of \eqref{equ:chem} is immediately recovered as the limit $M\to\infty$ of the left-hand side of~\eqref{equ:weak_chem}, again after integration in time and using \eqref{equ:mu}.
The right-hand sides of \eqref{equ:weak_chem} and \eqref{equ:chem} are composed of three contributions
associated with the quantities $\partial_\psi\Wel$, $\partial_{\psi}\Wpf$
and $\partial_{\nabla\psi}\Wpf$.
In virtue of~\eqref{def-W_double},
one can pass to the limit in the latter two terms
using the strong convergence \eqref{o-special-2}
and Sobolev embeddings.
However, the limit passage in the term associated with $\partial_{\psi}\Wel$, which is given in~\eqref{eq:Wder.psi}, is more complicated.
First, by the compact embedding $W^{2,\beta}(\Omega;\R^d)\Subset C^{1}(\overline{\Omega};\R^d)$, and by the assumptions on $g$ in \eqref{equ:alles_g-OO},
we conclude that $(\overline{\bF}_{{\mathrm e},M})_{M\in\N}$ is bounded in $L^{\infty}(\Omega;\R^{d\times d})$. Moreover, \eqref{equ:detbound} yields a pointwise lower bound for $\det(\nabla\by_M(x))$ for $x\in\Omega$. 
Hence, the dominated convergence theorem leads to
\begin{align*}
    \int_{\Omega_T} {\partial_{\psi}\Wel\left(\overline{\bF}_{{\mathrm e},M}\right)}\zeta\dxdt\to\int_{\Omega_T}{\partial_{\psi}\Wel\left({\bFg}_{{\mathrm e}}\right)}\zeta\dxdt\,
\end{align*}
as $M\to\infty$.
Summing up the individual components of the right-hand side, we conclude the proof of \eqref{equ:chem}. 

Next we perform the limit passage from \eqref{equ:weak_elast} to \eqref{equ:elast}, after integration over $[0,T]$. 
By~\eqref{equ:y_stark} we know $\overline{\bF}_{M}\to\bFg$ 
in $L^r(0,T;L^r(\Omega;\R^{d\times d}))$,
and \Cref{lemma:g_conv_NEO} yields ${g}^{-1}_M\to g^{-1}$ in $L^{\overline{q}}(\Omega)$ for $\overline{q}\in[1,6)$.
Hence, we find an a.e.~pointwise convergent (not relabled) subsequence 
such that $\overline{\bF}_{{\mathrm e},M}\to\bFg_{\mathrm e}$,
and thus 
$\left(\partial_{\bF}\Wel(\overline{\bF}_{{\mathrm e},M})\right)_M$
is pointwise convergent a.e. 
Using that $g$ and $\frac{1}{g}$ are bounded functions and that $(\bF_M)$ is uniformly bounded by~\eqref{equ:y_weak}, 
we observe that $\left(\partial_{\bF}\Wel(\overline{\bF}_{{\mathrm e}_M})\right)_M$
is uniformly bounded, and by dominated convergence we conclude 
for any $\bw\in L^{\beta}(0,T;W_{\DBC}^{2,\beta}(\Omega;\R^d))$
that
\begin{align*}
    \int_{\Omega_T}\partial_{\bF}\Wel(\overline{\bF}_{{\mathrm e},M}):\nabla\bw \dxds
    \to\int_{\Omega_T}\partial_{\bF}\Wel\left(\bFg_{{\mathrm e}}\right):\nabla\bw\dxds\quad\text{as }M\to\infty\,.
\end{align*}
Further we use the quadratic structure of the viscous potential $\hat{V}$, see~\ref{Ass5}, and the uniform boundedness of 
$(\overline{\bF}_M)_M\subset L^{\infty}(0,T;L^{\infty}(\Omega;\R^{\dxd}))$
due to~\eqref{equ:y_weak},
to obtain
\[
\begin{aligned}
\big\lvert\partial_{\dot{\bF}}V(\underline{\bF}_M,\partial_t{\hat{\bF}}):\nabla\bw\big\rvert 
&= \big\lvert (\partial_t{\hat{\bF}}^T\underline{\bF}_M+\underline{\bF}_M^T\partial_t{\hat{\bF}}):\bD(\underline{\bF}_M^T\underline{\bF}_M)(\nabla\bw^T\underline{\bF}_M+\underline{\bF}_M^T\nabla\bw)\big\rvert
\\
&\leq C |\partial_t{\hat{\bF}}| \,|\nabla \bw|
\leq C \big(|\partial_t{\hat{\bF}}|^2 +|\nabla \bw|^2\big).
\end{aligned}
\]
As $(\partial_t\hat\bF_M)_M\subset L^2(0,T;L^2(\Omega;\R^{\dxd}))$ 
is bounded due to~\eqref{equ:y_derivative}, 
for test functions $\bw\in L^2(0,T;W_{\Gamma_0}^{2,\beta}(\Omega;\R^d))$ it thus follows
\begin{align*}
\int_{\Omega_T}\partial_{\dot{\bF}}V(\underline{\bF}_M,\partial_t\hat{\bF}_M,\underline{\psi}_M):\nabla\bw \dxds
\to\int_{\Omega_T}\partial_{\dot{\bF}}V\left(\bF,\partial_t\bFg,\psi\right):\nabla\bw\dxds\,
\end{align*}
as $M\to\infty$ by Pratt's theorem and the continuity of $\bD$
assumed in~\ref{Ass5}.
The convergence of the hyperelasticity terms in \eqref{equ:weak_elast} follows 
by dominated convergence,
using the upper bound~\eqref{equ:prop-R_22} 
and the strong convergence of $(\nabla^2\bv_M)_M$ from~\eqref{o-special-1}.
In total, we thus obtain \eqref{equ:elast}.

Similarly, the energy-dissipation inequality~\eqref{equ:energy_O}
is derived from 
the discrete energy-dissipation inequality \eqref{equ:weak_energy} for the interpolants
by a limit procedure. 
To pass to the limit inferior on the left-hand side of~\eqref{equ:weak_energy}, 
we use the convergence properties in \eqref{equ:conv_weak}
and employ the lower semiconituity of $\fe$ (by \Cref{lemma:Coercive_LowerSemi}) and of the norm
in $L^2((0,t)\times\Omega)$.
For the viscous term associated to $V$,
we use the quadratic structure assumed in~\ref{Ass5} 
and the strong convergence of $(\underline{\bF}_M)$
induced by~\eqref{equ:y_stark}.
Concerning the right-hand side of~\eqref{equ:weak_energy},
we combine that $\by_M(t)\rightharpoonup\byg(t)$ in $\bfY$ and  $\psi_M(t)\rightharpoonup\psig(t)$ in $\bfZ$ 
by \eqref{equ:y_weak_ptw} and \eqref{equ:cong_psi_l2}
with the convergence of the corresponding free energy
stated in~\eqref{eq:conv-energy}
to conclude
\[
    \lim_{M\to\infty}
    \partial_t\fe (t,\underline{\by}_M(t),\underline{\psi}_M(t))
    =\partial_t\fe (t,\byg(t),\psig(t))
\]
from Lemma~\ref{lem:Fg.conv}.
As $\partial_t\fe (t,\underline{\by}_M(t),\underline{\psi}_M(t))$
is uniformly bounded due to the uniform continuity
stated in Lemma~\ref{OO_lemma:derivative_estimate},
we can pass to the limit on the right-hand side of~\eqref{equ:weak_energy}.
In total, this leads to the energy-dissipation balance~\eqref{equ:energy_O}
and completes the proof. 
\end{proof}

\section*{Acknowledgements}
The second author was partially supported by 
the German Research Foundation (DFG) within the DFG Priority
Program SPP 2171 \textit{Dynamic Wetting of Flexible, Adaptive, and Switchable Substrates} 
through the project no.~422786086.
The authors thank for helpful discussion
with Willem van Oosterhout about analytic aspects
and with
Dirk Peschka
about the modeling.
                                                     
%%%%%%%%%%%%%%%%%%%%%%%%%%%%%%%%%%%%%%%%%%%%%%%%%%%%%%%%%%%%%%%%%%%%%%

%%%%%%%%%%%%%%%%%%%%%%%%%%%%%%%%%%%%%%%%%%%%%%%%%%%%%%%%%%%%%%%%%%%%%%

\end{document}